\documentclass[a4paper,10pt]{article}
\usepackage{etex}
\usepackage[utf8x]{inputenc}
\usepackage{cite}

\usepackage[english]{babel}

\usepackage[lmargin=1.2in,rmargin=1.2in,tmargin=1in,bmargin=1in]{geometry}

\usepackage{amsmath}	
\usepackage{amssymb}
\usepackage{amsthm}
\usepackage{leftidx,tensor}
\usepackage{mathtools}

\usepackage{float}
\usepackage{graphicx}
\usepackage{pst-all}
\usepackage{color}
\usepackage{tikz}
\usepackage{xcolor}
\usepackage[all,cmtip]{xy}
\usetikzlibrary{calc,3d}
\usepackage{amsmath,amssymb,amsfonts}
\usepackage{pgfplots}
\pgfplotsset{compat=newest}
\usepackage{hyperref}

\usepackage{setspace}
\usepackage{subfig}
\usepackage{caption}
\usepackage{babel}

\newtheorem{theorem}{Theorem}[section]
\newtheorem{corollary}[theorem]{Corollary}
\newtheorem{lemma}[theorem]{Lemma}
\newtheorem{proposition}[theorem]{Proposition}

\newtheorem{definition}[theorem]{Definition}
\newtheorem{remark}[theorem]{Remark}

\title{Quantum Homology for Lagrangian Cobordism}
\author{Berit Singer}
 \date{February 2, 2016}

\begin{document}

\maketitle

\begin{abstract}
We extend the definition of Lagrangian quantum homology to monotone Lagrangian cobordism and establish its general algebraic
properties.
In particular we develop a relative version of Lagrangian quantum homology associated to a cobordism relative to a part of its boundary and study relations of this invariant to the ambient quantum homology.
\end{abstract}

\section{Introduction}\label{sec:introduction2}

The purpose of this paper is to develop a theory of quantum homology for Lagrangian cobordism, describe structures on it and explore relations to possible versions of the ambient quantum homology and the Lagrangian quantum homology of the ends of the cobordism.

Let $M^{2n}$ be a closed symplectic manifold and $L^n\subset M^{2n}$ a connected and closed Lagrangian submanifold.
The quantum homology $QH_*(M)$ of $M$ is additively the same as the singular homology of $M$.
We endow $QH_*(M)$ with the \emph{quantum product}, which is a deformation of the usual intersection product.
The Lagrangian quantum homology $QH(L)$ of a Lagrangian submanifold $L$ is the homology of the so called \emph{pearl complex}.
The pearl complex was first suggested by Fukaya in~\cite{fukaya} and by Oh in~\cite{oh} and later implemented by Biran and Cornea in~\cite{LQH} and further developed in~\cite{rigidity} and~\cite{LagrTop}.
We refer the reader to ~\cite{LQH},~\cite{rigidity} and~\cite{LagrTop} for details.
The pearl complex is defined by counting elements in $0$-dimensional moduli spaces of so called \emph{pearly trajectories}.
Roughly speaking, these pearly trajectories are Morse trajectories, where some points are replaced by pseudo-holomorphic disks.
In Section~\ref{sec:pearlcomplex} we briefly recall the main ideas for the construction 
of this chain complex.

Following~\cite{LagrCobordism} a \emph{Lagrangian cobordism} between two families $\{ L_1^-, \cdots , L_k^- \}$ and $\{ L_1^+, \cdots , L_l^+ \}$ of Lagrangian submanifolds in $M$, is a Lagrangian submanifold $V\subset \mathbb{R}^2\times M$ with $k+l$ cylindrical ends. 
The negative ends are identified with $(-\infty,  0]\times \{a_i^-\} \times L_j^-$ for $j=1, \cdots r_-$ and the positive ends identified with $[1,\infty)\times \{a_i^-\} \times L_i^+$ for $i=1, \cdots r_+$.
For the precise definition of Lagrangian cobordism see section~\ref{sec:LC} 
or~\cite{LagrCobordism}.
The main purpose of this paper is to define and study the quantum homology $QH(V,S)$ of a Lagrangian cobordism $V$ relative to a part of its boundary $S\subset \partial V$.
One of the possible applications of such a theory would be to study relations between symplectic invariants of different Lagrangians $L_-^i$ and $L_+^j$ that occur as different boundary components of the same Lagrangian cobordism.
In addition, there are natural maps describing the relation between the quantum homology of the total space of a Lagrangian cobordism and those of its ends. 
These maps preserve some of the structures on the quantum homologies, for example the ring structure.
These facts turn out to be helpful in order to calculate invariants of Lagrangian submanifolds.
For instance Biran and Membrez used quantum homology calculations in~\cite{cedric} to show that the discriminant of certain Lagrangian submanifolds are preserved under cobordism.
It is also natural to look for relations between the quantum homologies of Lagrangian 
submanifolds and the quantum homology of the ambient manifold.

The construction of Lagrangian quantum homology for cobordisms has already been outlined in~\cite{LagrCobordism}.
Let $\mathcal{R}$ be a ring, and denote by $\Lambda:=\mathcal{R}[t,t^{-1}]$ the ring of Laurent polynomials in the variable $t$, where $deg(t):=-N_V$ is minus the minimal Maslov number $N_V$ of $V$.
Unless $Char(\mathcal{R})=2$, we assume that $V$ is spin and in this case we fix a spin structure on $V$.
We endow the ends of the cobordism with the spin structures obtained by restricting the spin structure of $V$.
Denote $\tilde{M}^{2n+2}:=\mathbb{R}^2\times M$ and fix the symplectic form $\omega:=\omega_{\mathbb{R}^2} \oplus \omega_M$ on $\tilde{M}$, 
where $\omega_{\mathbb{R}^2}$ is the symplectic form on $\mathbb{R}^2$ given by $dx \wedge dy$ for $(x,y)\in \mathbb{R}^2$.
Let $\pi: \tilde M \rightarrow \mathbb{R}^2$ be the projection.
We use the notation $V|_A:=V \cap \pi^{-1}(A)$ and $\tilde{M}|_A:=\tilde{M}\cap 
\pi^{-1}(A)$ for a subset $A\subset\mathbb{R}^2$.
In particular $\tilde{M}|_{[0,1]\times \mathbb{R}}=[0,1]\times \mathbb{R} \times M$.
Fix a Riemannian metric $\tilde{\rho}$ and an $\omega$-compatible almost complex structure 
$\tilde{J}$ on $\tilde{M}$.
This induces a Riemannian metric and an almost complex structure on $\tilde{M}|_{[0,1]\times \mathbb{R}}$, which will also be denoted by $\tilde{\rho}$ and $\tilde{J}$.
Consider now $V|_{[R_-,R_+]\times \mathbb{R}}$, where $R_-<0$ and $R_+>0$ are such that 
$V$ is cylindrical outside of $[R_-,R_+]\times \mathbb{R}\times M$.
Let $S\subset \partial V|_{[R_-,R_+]\times \mathbb{R}}$ be the union of some connected components of the boundary of $V|_{[R_-,R_+]\times \mathbb{R}}$.
Let $\tilde{f}$ be a Morse function on $V|_{[R_-,R_+]\times \mathbb{R}}$, such that the negative gradient of $\tilde{f}$ points outside along $S$ and inside along $\partial V|_{[R_-,R_+]\times \mathbb{R}} \setminus S$.
Moreover, we assume some additional condition on the almost complex structure near the 
boundary $\partial V|_{[R_-,R_+]\times \mathbb{R}}$, which are described in detail 
later in this paper.
With this data we are able to define a relative quantum homology for the the pair $(V,S)$.
The following theorem describes the existence of algebraic structures on the quantum homology of Lagrangian cobordism.
These structures are similar to the ones known for the usual Lagrangian quantum homology described in~\cite{QuantumStructures}.
Recall that the dimension of $V$ is $n+1$ and that of $\tilde{M}$ is $2n+2$.
\begin{theorem}\label{thm:QuantumStructures_LQHforLC}
 For a generic choice of the triple $(\tilde{f},\tilde{\rho},\tilde{J})$ there exists a chain complex
\begin{equation*}
 C(V,S;\tilde{f},\tilde{\rho},\tilde{J})=(\mathcal{R} \langle Crit(\tilde{f}) \rangle \otimes \Lambda, d)
\end{equation*}
with the following properties:
\begin{enumerate}
 \item The homology $QH_*(V,S)$ of the chain complex is independent of the choices of $\tilde{f}$, $\tilde{\rho}$ and $\tilde{J}$.
 \item If $S=\emptyset$ there exists a canonical, degree preserving augmentation 
\begin{equation*}
 \epsilon_V:QH_*(V) \rightarrow \Lambda,
\end{equation*}
which is a $\Lambda$-module map and satisfies an additional property that is detailed in point~\ref{itm:duality} below.
  \item
\begin{enumerate}
 \item\label{itm:prod} There exists a $\Lambda$-bilinear product 
\begin{equation*}
\begin{array}{ccc}
  *:QH_i(V,S)\otimes QH_j(V,S) & \rightarrow & QH_{i+j-(n+1)}(V,S) \\
  \alpha \otimes \beta & \mapsto & \alpha * \beta,\\
\end{array}
\end{equation*}
which endows $QH_*(V,S)$ with the structure of a (possibly non-unital) ring.
Moreover, if $S=\partial V$, then the $QH_*(V,\partial V)$ is a ring with unit.
\item Let $(G,\cap)$ be the monoid, where $G=\{\text{subsets of } \pi_0( \partial 
V) \}$ and $ \cap: G \times G\rightarrow G: (S,S') \mapsto S\cap S'$ is the intersection, 
with unit $\partial V$.
There exists a $\Lambda$-bilinear product 
\begin{equation*}
\begin{array}{ccc}
  *:QH_i(V,S)\otimes QH_j(V,S') & \rightarrow & QH_{i+j-(n+1)}(V,S\cap S') \\
  \alpha \otimes \beta & \mapsto & \alpha * \beta,\\
\end{array}
\end{equation*}
which extends the product in~\ref{itm:prod} and endows $\bigoplus_{S\in G} QH_*(V,S)$ with the structure of a unital graded ring over the monoid $G$.
\end{enumerate}

\item Consider the sets $T:=[0,1]\times \mathbb{R} \subset \mathbb{R}^2$and $R:=[0,1]\times [-k,k] \subset \mathbb{R}^2$, where $k>0$ is large enough so that $V|_{\mathbb{R}\times[-k,k]}=V$.
Denote $\tilde{M}_T:=\tilde{M}|_T$ and $\tilde{M}_R:=\tilde{M}|_R$. 
Then there exist ambient quantum homologies $QH_*(\tilde{M}_T,\partial \tilde{M_T})$ and $QH_*(\tilde{M}_R,\partial\tilde{M_R})$, the latter having the structure of a unital ring and $QH_*(\tilde{M}_T,\partial \tilde{M_T})$ is a module over $QH_*(\tilde{M}_R,\partial\tilde{M_R})$.
There exists a bilinear map
\begin{equation*}
 *:QH_i(\tilde{M}_R,\partial \tilde{M_R}) \otimes QH_j(V,S) \rightarrow 
QH_{i+j-(2n+2)}(V,S),
\end{equation*}
which turns $QH_*(V,S)$ and $\bigoplus_{S\in G} QH_*(V,S)$ into a module over the ring 
$QH_*(\tilde{M}_R,\partial
\tilde{M_R})$.
Moreover $QH(V,S)$ is a (possibly non-unital) two-sided 
algebra over the ring $QH_*(\tilde{M}_R,\partial \tilde{M_R})$.
In particular, the rings $\bigoplus_{S\in G} QH_*(V,S)$ and $QH(V,\partial V)$ are 
unital two-sided algebras.
  \item\label{itm:inclusion} There exists a $QH_*(\tilde{M}_R,\partial \tilde{M_R})$- linear inclusion map
\begin{equation*}
 i_{(V,S)}:QH_*(V,S) \rightarrow QH_*(\tilde{M}_T, \partial \tilde{M}_T),
\end{equation*}
which extends the inclusion in singular homology and is determined by
\begin{equation}
 \langle h^*,i_L(x)\rangle=\epsilon(h*x),
\end{equation}
for every $x\in QH_*(L)$, $h\in QH_*(M)$. 
Here $(-)^*$ denotes the Poincar\'e dual and $\langle \cdot, \cdot \rangle$ the Kronecker pairing.
\item\label{itm:duality} There exists an isomorphism 
\begin{equation*}
\begin{array}{ccc}
  \eta:QH_i(V,S) & \rightarrow & QH^{(n+1)-i}(V,\partial V \setminus S),
\end{array}
\end{equation*}
where $QH^{(n+1)-i}(V,\partial V \setminus S)$ denotes the k'th cohomology ring 
associated to the cochain complex $C(V,S;\tilde{f},\tilde{\rho},\tilde{J})^*$.
The corresponding (degree $-(n+1)$) bilinear map 
\begin{equation*}
\begin{array}{ccc}
  \tilde{\eta}:QH_i(V,S)\otimes QH_{(n+1)-i}(V, S) & \rightarrow & \Lambda \\
  x\otimes y & \mapsto & [\eta(x)(y)],\\
\end{array}
\end{equation*}
satisfies the identity 
\begin{equation*}
 \tilde{\eta}(x\otimes y) = \epsilon_V(x*y).
\end{equation*}

\end{enumerate}
\end{theorem}
Parts of this theorem have already appeared in~\cite{LagrCobordism} with an outline of 
the proof. 
Below we elaborate more on these and also prove the new statements.
Our approach is based very much on the general theory of Lagrangian quantum homology from~\cite{LQH} and~\cite{LagrCobordism}.
Our next result describes a relation between the quantum homologies of $V$ relative to 
its boundary and that of its boundary.
\begin{theorem}\label{thm:longexact}
Let $S\subset \partial V$ be the union of some of the connected components of $\partial 
V$.
There exists a long exact sequence
\begin{equation*}
\xymatrix{
 \dots \ar[r]^{\delta_*} & QH_*(S) \ar[r]^{i_*} & QH(V)\ar[r]^{j_*} & QH_*(V,S)\ar[r]^{\delta_*} & QH_{*-1}(S)\ar[r]^{i_*} & \dots,
}
\end{equation*}
which has the following properties:
\begin{enumerate}
\item Suppose that $S=\partial V$. Let $e_{(V,\partial V)}$ and $e_L$ denote the unit of $QH(V,\partial V)$ and $QH(L)$ respectively. 
Then 
$$\delta_*(e_{(V,\partial V)})= \oplus_{i} e_{L_i^-}\oplus_{j} e_{L_j^+},$$
 where $\partial V= \coprod L_i^- \coprod L_j^+$.
\item  The map $\delta_*$ is multiplicative with respect to the quantum product $*$, namely 
\begin{equation}
 \delta_*(x*y)=\delta_*(x)*\delta_*(y) \ \ \forall x,y \in QH_*(V,S).
\end{equation}
\item  The product $*$ on $QH_*(V)$ is trivial on the image of the map $i_*$.
In other words, for any two elements $a$ and $b$ in $QH_*(S)$ we have that 
$$i_*(a)*i_*(b)=0.$$
\item The map $j_*$ is multiplicative with respect to the quantum product, namely
\begin{equation*}
 j_*(x*y)=j_*(x)*j_*(y)\ \ \forall x,y \in QH_*(V).
\end{equation*}
\item There exists a ring isomorphism $\Phi: QH_*(M) \rightarrow QH_{*+2}(\tilde{M}_R,\partial \tilde{M}_R)$. 
Recall also that $QH_*(\partial V)$ is a module over the ring $QH_*(M)$, and $QH_*(V)$, 
$QH_*(V,\partial V)$ are modules over the ring $QH_{*+2}(\tilde{M}_R,\partial 
\tilde{M}_R)$.
We then have the following identities:
\begin{enumerate}
 \item[(i)] $i_*(a \ast x)= \Phi(a) \ast i_*(x)$, for every $x\in QH(\partial V)$ and every $a\in QH(M)$.
 \item[(ii)] $j_*(a\ast x)=a\ast j_*(x)$, for every $x\in QH(V)$ and every $a\in QH(\tilde{M}_R,\partial \tilde{M}_R)$.
 \item[(iii)] $\delta_*(a\ast x)=\Phi^{-1}(a)\ast \delta_*(x)$, for every $x\in 
QH(V,\partial V)$ and every $a\in QH(\tilde{M}_R,\partial \tilde{M}_R)$.
\end{enumerate}
i.e. the maps in the long exact sequence are module maps over the ring 
$QH_*(\tilde{M}_R,\partial \tilde{M}_R)$ and $QH_*(M)$.
\end{enumerate}
\end{theorem}

\subsection*{Organization of the paper}
In the rest of Section~\ref{sec:introduction2} we first recall the construction of the pearl complex in the compact setting.
Then we explain the ingredients used to prove Theorem~\ref{thm:QuantumStructures_LQHforLC} and Theorem~\ref{thm:longexact} and give a short outline of their proofs. 
The second section is dedicated to a brief description of the setting we are working in.
It also includes an overview of the ambient quantum homology, Lagrangian cobordism and Lagrangians with cylindrical ends.
Sections~\ref{sec:LQHforLC} -~\ref{sec:inclusion} are dedicated to the proof Theorem~\ref{thm:QuantumStructures_LQHforLC}.
In Section 9 we discuss the orientations of the moduli spaces and prove 
Theorem~\ref{thm:longexact} in detail.
In the last section we give an example and compute the various quantum structures for it.

\subsection{The Pearl Complex}\label{sec:pearlcomplex}

For the sake of readability we include in this section a short overview of the construction of the pearl complex as it is defined in~\cite{QuantumStructures}.
Let $Q$ be a closed, monotone Lagrangian submanifold of a connected symplectic manifold $(X,\omega)$ and assume that $Q$ is endowed with a spin structure.
As before, let $\mathcal{R}$ be a ring, and define the ring $\Lambda:=\mathcal{R}[t,t^{-1}]$ of Laurent polynomials in the variable $t$, where the degree of t is $|t|:=-N_Q$ is minus the minimal Maslov number $N_Q$ of $Q$.
(For the definition of monotonicity see Chapter~\ref{sec:setting}.)
If the ground ring $\mathcal{R}$ has characteristic $2$ we do not need to assume the existence of a spin structure.
We fix a Morse function $f:Q\rightarrow \mathbb{R}$ on $Q$, a Riemannian metric $\rho$ on $Q$ as well as an almost complex structure $J$ on $X$.
Given a generic triple $(f,\rho,J)$ we define a complex 
\begin{equation}\label{prl}
 C(f,\rho,J):=\mathcal{R}\langle Crit(f)\rangle \otimes \Lambda,
\end{equation}
here the generators are the critical points of $f$.
The grading in~\ref{prl} is induces from both factors, where $x\in Crit(f)$ is graded by the Morse index of $f$.
The differential of the chain complex is defined using moduli spaces of pearly trajectories.
Given two critical points $x$ and $y$ in $Q$, a pearly trajectory between them is the flow of the negative gradient $-\nabla f$ with finitely many points replaced by non-constant pseudo-holomorphic disks
\begin{equation}
 u_i:(D,\partial D) \rightarrow (X,Q),
\end{equation}
with boundary in $Q$ and with total class $\sum [u_i]=A\in H^D_2(X,Q)$.
Here $H^D_2(X,Q)$ stands for the image of $\pi_2(X,Q)$ in $H_2(X,Q)$ under the Hurewicz homomorphism.
The following figure shows two pearly trajectories connecting two critical points $x$ and $y$ of $f$. 
The first one has $k$ pseudo-holomorphic disks with total class $A\neq0$, the second one has no pseudo-holomorphic disks and thus class $A=0$.

%
%
%
%
%
%
%
%
%
%
\begin{figure}
 \centering
\includegraphics{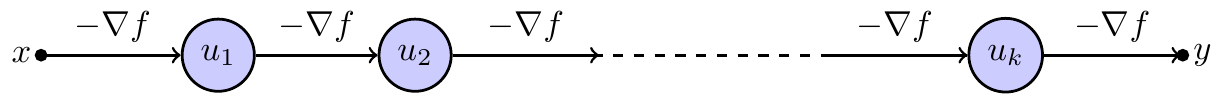}
\centering
\includegraphics{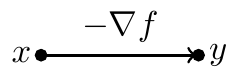}
\end{figure}

These moduli spaces are called moduli spaces of pearly trajectories.
Denote by $\mathcal{P}_{prl}(x,y,A;f,\rho,J)$ the space of unparametrized pearly trajectories between the critical points $x$ and $y$ of $f$ and with total class $A$.
Define $\delta_{prl}(x,y,A):=|x|-|y|+\mu(A)-1$ the virtual dimension of $\mathcal{P}_{prl}(x,y,A;f,\rho,J)$.
In~\cite{QuantumStructures} it is shown that when $\delta_{prl}(x,y,A)\leq 1$, the space $\mathcal{P}_{prl}(x,y,A;f,\rho,J)$ is a smooth manifold of dimension $\delta_{prl}(x,y,A)$.
The spin structure on $Q$ allows to orient the moduli spaces of pearly trajectories.
If $\delta_{prl}(x,y,A)=0$ one can show that $\mathcal{P}_{prl}(x,y,A;f,\rho,J)$ is compact.
The differential $d$ of the complex~(\ref{prl}) is now given as follows. 
For a critical point $x$ define
\begin{equation}
 d(x)=\sum \sharp \mathcal{P}_{prl}(x,y,A;f,\rho,J) y t^{\mu(A)/N_Q},
\end{equation}
where the sum runs over all $y$ and $A$ such that $\delta_{prl}(x,y,A)=0$.
We extend $d$ linearly over $\Lambda$.
Since $\mathcal{P}_{prl}(x,y,A;f,\rho,J)$ is compact the number $\sharp \mathcal{P}_{prl}(x,y,A;f,\rho,J)$ is finite and thus the map $d$ is well-defined.

To show that $d\circ d=0$ we describe the boundary of the compactification of the $1$-dimensional moduli spaces of pearly trajectories.
This can be expressed by $0$-dimensional moduli spaces of pearly trajectories.
Let $\mathcal{P}$ be a $1$-dimensional moduli space of pearly trajectories.
In order to compactify $\mathcal{P}$ we need to add limit trajectories of the following types.
\begin{enumerate}\label{boundary}
 \item[(1)] One of the flow lines of the negative gradient of the Morse function breaks at a critical point.
 \item[(2)] One of the flow lines of the negative gradient of the Morse functions contracts to a constant point.
 \item[(3)] Bubbling of a pseudo-holomorphic disc of the pearly trajectory.
\end{enumerate}
Then the compactification $\overline{\mathcal{P}}$ is obtained from $\mathcal{P}$ by adding all possible limit trajectories and we have:
\begin{equation*}
 \partial\overline{\mathcal{P}}=\{ \text{elements of type (1)} \} \cup \{ \text{elements of type (2)} \}\cup \{ \text{elements of type (3)} \}.
\end{equation*}
For a detailed description of the compactification the reader is referred to~\cite{QuantumStructures}.
Computing $d\circ d$ we can see that the coefficient of some critical point $z$ in the sum $d\circ d(x)$ is the number of two concatenated pearly trajectories, where one goes from $x$ to a critical point $y$ and the other goes from $y$ to $z$. 
By a gluing argument, they are in one to one correspondence with elements of type (1). 
A dimension argument and the fact that the Maslov index of pseudo-holomorphic disks is bounded by the minimal Maslov number $N_Q\geq2$, proves that elements of type (2) and (3) sum to zero, when counted with orientations. 
Hence the coefficient of $z$ in the sum $d\circ d(x)$ is equal to the number of boundary points of $\overline{\mathcal{P}}$ counted with signs, i.e. $\sharp \partial \overline{\mathcal{P}}$.
But since $\overline{\mathcal{P}}$ is a compact $1$-dimensional manifold with boundary we have $\sharp \partial \overline{\mathcal{P}}=0$.
This implies that $d\circ d=0$.

The definition of the chain complex above depends on the choice of the triple $(f,\rho,J)$. 
However, the quantum homology turns out to be independent of such a choice. 
In order to prove this we need to show the existence of chain maps $\phi_{\mathcal{D}}^{\mathcal{D'}}$ for any choice of two generic triples $\mathcal{D}=(f,\rho,J)$ and $\mathcal{D'}=(f',\rho',J,)$.
The chain map $\phi_{\mathcal{D}}^{\mathcal{D'}}: C(\mathcal{D})\rightarrow C(\mathcal{D'})$ then induces a canonical isomorphism on the homology level.
Moreover, the system of such chain maps is compatible with composition and $\phi_{\mathcal{D}}^{\mathcal{D}}=Id$.
This would proves the invariance of $QH(Q)$.
The chain map $\phi_{\mathcal{D}}^{\mathcal{D'}}$ depends on the following data: a generic homotopy $J_t$ of almost complex structures from $J$ to $J'$ and a Morse cobordism $(F,G)$ from the Morse pair $(f,\rho)$ to the pair $(f',\rho')$.
One then defines the so called comparison moduli spaces $\mathcal{P}_{comp}$, which are a small modification of the usual pearl moduli spaces.
For a generic choice of $(J_t,F,G)$ they form smooth manifolds of dimension $|x|-|y|+\mu(\lambda)$.
Counting elements in the $0$-dimensional comparison moduli spaces defines the chain maps $\phi_{\mathcal{D}}^{\mathcal{D'}}$. 
For a precise description we refer the reader to~\cite{QuantumStructures}.

\subsection{Outline of the Proof of 
Theorem~\ref{thm:QuantumStructures_LQHforLC}}\label{sec:OutlineofProof}

In this section we describe the main ingredients in extending the notion of Lagrangian quantum homology to Lagrangian cobordism.
As in the compact case we would like to define a chain complex, where the differential is given by counting elements in the moduli spaces of pearly trajectories. 
The problems that arise when defining the pearl complex for the quantum homology are due to the fact that Lagrangian submanifolds with cylindrical ends are non-compact or alternatively we can be viewed as manifolds with boundary.
Note however that the non-compactness arises only at the cylindrical ends of the Lagrangians.

Away from the boundary of the Lagrangian cobordism transversality and compactness of the pearly moduli spaces follow in the same way as for the setting of closed Lagrangians. 
Since the ends of the cobordism are cylindrical, we may choose the almost complex structure to be split near the boundary of the cobordism.
In other words, we may assume that there exists a compact set $K\subset \mathbb{R}^2$, such that outside of $K\times M$ the cobordism is cylindrical and the almost complex structure has the form $i\oplus J$ for some almost complex structure $J$ on $M$.
Here $i$ is the standard complex structure on $\mathbb{C}\cong \mathbb{R}^2$.
Pseudo-holomorphic disks mapping to the neighbourhood of the boundary of the cobordism will turn out to be constant under the projection $\pi:M\times \mathbb{R}^2 \rightarrow \mathbb{R}^2$, hence their image lies in one fiber of $\pi$.
This is a consequence of the open mapping theorem and it is the content of the following lemma from~\cite{LagrCobordism}.
\begin{lemma}\label{lem:curve_openmapping}
Let $u: \Sigma \rightarrow \tilde{M}$ be a $\tilde{J}$-holomorphic curve, where $\Sigma$ is either $S^2$ or the unit disc $D$ with $u(\partial D) \subset V$.
Then either $\pi \circ u$ is constant or its image is contained in $K$, i.e. $\pi \circ u(\Sigma) \subset K$.
\end{lemma}

Since this lemma plays a central role in our argument we include the proof following the lines of~\cite{LagrCobordism}.
\begin{proof}
Note that $\pi\circ u$ is bounded since $\Sigma$ is compact.
Suppose now that $\pi \circ u(\Sigma) \nsubseteq K$.
Assume by contradiction that $\pi \circ u$ is not constant.
The set $\mathbb{C}\setminus ( K \cup \pi(V))$ is a union of unbounded, connected, open subsets of $\mathbb{C}$. 
Let $W$ be one of these connected open subsets.
Notice that $\pi\circ u(\Sigma)\cap W=\pi\circ u(int(\Sigma))\cap W=\overline{\pi\circ u(\Sigma)}\cap W$.
By the open mapping theorem, the image $\pi\circ u(int(\Sigma))\cap W$ of the open connected set $int(\Sigma)$ under the holomorphic map $\pi \circ u$ must be open in $W$.
On the other hand $\overline{\pi\circ u(\Sigma)}\cap W=\pi\circ u(\Sigma)\cap W$ is closed in $W$.
Since $W$ is connected it follows that $\pi\circ u(\Sigma)\cap W=W$.
This is a contradiction, since $W$ is unbounded.
\end{proof}
In particular, the previous lemma ensures that the pearly trajectories do not reach the boundary of the cobordism and thus the arguments used in the compact case work also in our setting.
Analogous arguments will also apply to the moduli spaces used to define the quantum product, the module structure and the inclusion.
Lemma~\ref{lem:curve_openmapping} therefore plays a key role in the proof of Theorem~\ref{thm:QuantumStructures_LQHforLC}.

Finally, we need to ensure that the moduli spaces of $(i\oplus J)$-holomorphic disks, which are constant in the $\mathbb{R}^2$-factor, are manifolds.
This guarantees transversality of the moduli spaces near the boundary.
The following lemma can be proven by a standard calculation. 
We omit the details.
\begin{lemma}\label{automatictransversality}
Let $u:\Sigma \rightarrow \mathbb{C}: z \mapsto p$ be a constant map, where $\Sigma$ is either the unit disk $D$ or the sphere $S^2$. 
Then the linearization $D_u$ of the $\overline{\partial}$ map at the curve $u$ is surjective. 
Hence, the space of $(i\oplus J)$-holomorphic curves $u:(\Sigma,\partial \Sigma)\rightarrow (\tilde{M},V)$ with constant projection to the $\mathbb{R}^2$-factor is a manifold of dimension equal to the index of the Fredholm operator $D_{u'}$, where $u=(p,u')$ and $u'$ is a $J$-holomorphic curve in $M$.
\end{lemma}

\subsection{Outline of the Proof of Theorem~\ref{thm:longexact}}\label{sec:OutlineofProof2}

Consider the compact version $V|_{[R_-,R_+]}$ of a cobordism $V$, where $R_-$ and $R_+$ are such that $V$ is cylindrical outside of $\pi^{-1}([R_-,R_+]\times \mathbb{R})$.
To prove the existence of a long exact sequence
\begin{equation*}
\xymatrix{
 \dots \ar[r]^{\delta_*} & QH_*(S) \ar[r]^{i_*} & QH(V) \ar[r]^{j_*} & QH_*(V,S) \ar[r]^{\delta_*} & QH_{*-1}(S)\ar[r]^{i_*} & \dots
}
\end{equation*}
we define a special Morse function $\tilde{f}$ on an small extension $V^{\epsilon}:=V|_{[R_- -\epsilon,R_++\epsilon]}$.
The description of such a function can also be found in~\cite{LagrCobordism}.
Let us identify the cylindrical ends of the cobordism with $[R_- -\epsilon,  R_-]\times \{a_i^-\} \times L_i^-$ and $[R_+,R_++\epsilon]\times \{a_j^+\} \times L_j^+$. 
We assume that on the cylindrical ends of $V$ the almost complex structure $\tilde{J}$ splits into $i\oplus J$ for some $J$ on $M$ and $i$ the standard complex structure on $\mathbb{R}^2$.
Likewise we assume that on the cylindrical ends of $V$ the metric $\rho$ is of the form $\rho^{\pm}\oplus \rho_M$, for some metric $\rho^{\pm}$ on $\coprod L_i^-$ and $\coprod L_j^+$.
Recall that $S$ is a collection of connected components of the boundary of $V|_{[R_-,R_+]}$. 
Let $S^{\epsilon}$ the part of the end of $V|_{[R_- - \epsilon,R_+ +\epsilon]}$ corresponding to $S$.
More general we define $S^{\eta}$ the part of the boundary of $V|_{[R_- - \eta,R_+ +\eta]}$ corresponding to $S$.
Fix an end of the cobordism belonging to $S^{\epsilon}$, say $[R_- -\epsilon,  R_-]\times \{a_i^-\} \times L_i^-$. 
On the end $[R_- -\epsilon,  R_-]\times \{a_i^-\} \times L_i^-$ we define $\tilde{f}$ to be the sum of two Morse functions, $f_i^-$ on $L_i^-$ and $\sigma_i^-$ on $[R_- -\epsilon,  R_-]$.
If the end belongs to the set $S^{\epsilon}$ we choose $\sigma_i^-$ such that it has a single maximum and no other critical points.
Likewise, if the end does not belong to $S^{\epsilon}$ we choose $\sigma_i^-$ with a single minimum and no other critical points.
We may assume that the single critical point of $\sigma_i^-$ lies at $R_- -\epsilon/2$ and that $\sigma_i^-$ is linear on the set $[R_- -\epsilon,  R_- - 3\epsilon/4]$.
Similarly we choose functions $\sigma_j^+$ on the positive ends with a single maximum, if the end belongs to $S^{\epsilon}$ and a single minimum otherwise.
\begin{figure}
\centering
 \includegraphics{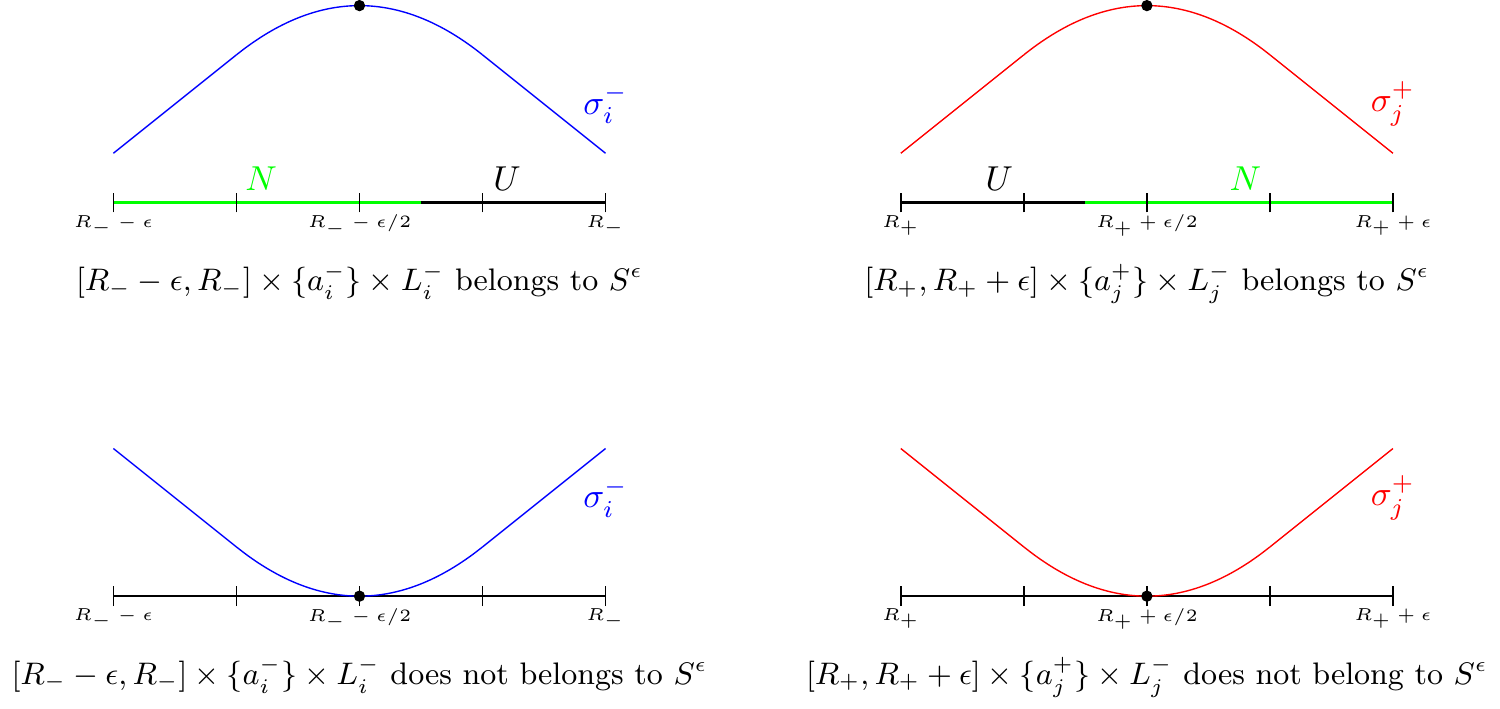}
\caption{The functions $\sigma_i^-$ and $\sigma_j^+$.}\label{sigma}
\end{figure}

Figure~\ref{sigma} illustrates the choice of the functions $\sigma_i^-$ on the left and $\sigma_j^+$ on the right.
This choice of the functions $\sigma_i^-$ and $\sigma_j^+$ ensures that the negative gradient $-\nabla f$ points outside along $S^{\epsilon}$ and inside along $\partial V^{\epsilon}\setminus S^{\epsilon}$.
Let $N$ denote the neighborhood of $S^{\epsilon}\subset \partial V^{\epsilon}$ given by 
\begin{equation}
 \coprod_{\{(R_- -\epsilon  ,i)\} \times L_i^- \subset S^{\epsilon}} [R_- -\epsilon,  R_- -\epsilon/4]\times \{a_i^-\} \times L_i^-\coprod_{\{(R_+ +\epsilon ,j)\} \times L_j^- \subset S^{\epsilon}} [R_++\epsilon/4,  R_+ +\epsilon]\times \{a_j^+\} \times L_j^- .
\end{equation}
In other words, $N$ is a union of cylindrical ends, corresponding to the set $S^{\epsilon}$.
Define $U:=V^{\epsilon}\setminus N$ and notice that $S^{\epsilon/2}\subset N$.
Restricting the function $\tilde{f}$ to the set $U$ we can see that now the negative gradient $-\nabla \tilde{f}|_{U}$ points inside along the whole boundary of $U$.
Notice that the critical points of $\tilde{f}$, $\tilde{f}|_{U}$ and $\tilde{f}|_{S^{\epsilon/2}}$ generate the complexes $C_*((V^{\epsilon},S^{\epsilon});\tilde{f},\tilde{J}) $, $C_*(U;\tilde{f}|_U,\tilde{J})$ and $C_*(S^{\epsilon/2};\tilde{f}|_{S^{\epsilon/2}},J)$ respectively.
Therefore we obtain a short exact sequence of chain complexes:
\begin{equation}
\xymatrix{
0 \ar[r] & C_k(U;\tilde{f}|_U,\tilde{J}) \ar[r]^j \ar[d]^{d_U} & C_k((V^{\epsilon},S^{\epsilon});\tilde{f},\tilde{J}) \ar[r]^{\delta} \ar[d]^{d_{(V^{\epsilon},S^{\epsilon})}} & C_{k-1}(S^{\epsilon/2};\tilde{f}|_{S^{\epsilon/2}},J) \ar[d]^{d_{S^{\epsilon/2}}} \ar[r] & 0\\
0 \ar[r] & C_{k-1}(U;\tilde{f}|_U,\tilde{J}) \ar[r]^{j} & C_{k-1}((V^{\epsilon},S^{\epsilon});\tilde{f},\tilde{J}) \ar[r]^{\delta} & C_{k-2}(S^{\epsilon/2};\tilde{f}|_{S^{\epsilon/2}},J) \ar[r] & 0 \\
},
\end{equation}
where the map $j$ is given by the inclusion of critical points and the map $\delta$ is given by the restriction to the critical points of $\tilde{f}|_S$.
For a precise definition and a proof that this is a short exact sequence of chain complexes, see Sections~\ref{sec:LQHforLC} and~\ref{chap:longexact}.
By definition (see Section~\ref{sec:LQHforLC}) the Lagrangian quantum homology $QH(V,S)$ is the homology of the chain complex $C_*((V^{\epsilon},S^{\epsilon});\tilde{f},\tilde{J})$.Clearly the homology of the chain complex $C_*(S^{\epsilon/2};\tilde{f}|_{S^{\epsilon/2}},J)$ is isomorphic to Clearly the homology of the chain complex $C_*(S^{\epsilon/2};\tilde{f}|_{S^{\epsilon/2}},J)$ is isomorphic to $C_*(S;\tilde{f}|_{S},J)$.
This short exact sequence induces a long exact sequence in homology as asserted in Theorem~\ref{thm:longexact}.
Proving the other statements in Theorem~\ref{thm:longexact} amounts to comparing elements and orientations of the zero dimensional moduli spaces used for the definition of the operations on the various quantum homologies.

\section{Setting}\label{sec:setting}

In this paper we work with connected, closed, monotone Lagrangians $L\subset 
(M^{2n},\omega)$, where $(M,\omega)$ is a tame monotone symplectic manifold.
Denote by $H_2^D(M,L)$ the image of the Hurewicz homomorphism $\pi_2(M,L)\rightarrow H_2(M,L)$.
Thus, we have that
\begin{equation}\label{eq:monotonicity}
\begin{array}{ccc}
  \omega(A)=\tau \mu(A), \ \forall   A \in H^D_2(M,L),
\end{array}
\end{equation}
where $\tau>0$ is the monotonicity constant of $L\subset(M,\omega)$.
Additionally we assume that the minimal Maslov number
\begin{equation*}
 N_L:=min\{\mu(A) | \mu(A)>0, A \in \pi_2(M,L)\},
\end{equation*}
is at least $2$.
We often use the notation $\overline{\mu}:=\frac{\mu}{N_L}$.

Denote by $c_1\in H^2(M)$ the first Chern class of the tangent bundle of $M$.
In particular $(M,\omega)$ is spherically monotone with a constant $\eta >0$, such that 
$\omega(A)=\eta c_1(A)$ for every $A\in H_2^S(M)$, where $H_2^S(M)$ denotes the image of the Hurewicz map $\pi_2(M) \rightarrow H_2(M)$.
One can show that the two constants $\mu$ and $\eta$ are related by $\eta=2\mu$.
In particular, we have $N_L | 2 C_M$, where $C_M$ is the minimal Chern number on $2$-dimensional spherical classes.

\subsection{The Ambient Quantum Homology}\label{sec:QHandFH}

Let $\Lambda=\mathcal{R}[t,t^{-1}]$ be the graded ring of Laurent polynomials in the variable $t$ of $deg(t):=-N_L$.
The ambient quantum homology is defined as $QH_*(M):= H_*(M;\mathcal{R})\otimes \Lambda$, where the grading is induced by
the grading of $H_*(M)$ and $\Lambda$.
\begin{remark}
 For the definition of the ambient quantum homology one often uses the grading 
$deg(t):= -2C_M$. 
However, since $N_L | 2C_M$ our definition is only an extension of the quantum homology 
with the grading $deg(t):=-2C_M$.
\end{remark}
Let
\begin{equation*}
 \begin{array}{ccc}
  * : QH_{k}(M)\otimes QH_{l}(M) \ \rightarrow \ QH_{k+l-2n}(M),
 \end{array}
\end{equation*}
denote the \emph{quantum intersection product}.
For more details on this subject the reader is referred to~\cite{McDuffSalamon}.

\subsection{Lagrangian Cobordism and Lagrangians with Cylindrical Ends}\label{sec:LC}

We begin this section by recalling the definition of Lagrangian cobordism as in~\cite{LagrCobordism}.
\begin{definition}\label{def:cob}
 A Lagrangian cobordism between two families of closed Lagrangian submanifolds of $M$, say $\{ L_j^+ \}_{1 \leq j \leq r_+}$ and $\{ L_i^- \}_{1 \leq i \leq r_-}$, is a cobordism $(V, \coprod_i L_i^-, \coprod_j L_j^+)$ from $\coprod_j L_j^+$ to $\coprod_i L_i^-$, such that there exists a Lagrangian embedding $V^{n+1} \subset [0,1]\times \mathbb{R} \times M$, which is of the form
\begin{equation*}
\begin{array}{ccc}
  V|_{[0,\epsilon)\times \mathbb{R}} \ = \ \coprod_i ([0, \epsilon)\times \{a_i^-\})\times L_i^- \\
 V|_{(1-\epsilon,1] \times \mathbb{R}} \ = \ \coprod_j ((1-\epsilon, 1]\times \{a_j^+\})\times L_j^+,
\end{array}
\end{equation*}
for some $\epsilon >0$.
\end{definition}

Note that in Definition~\ref{def:cob}, the Lagrangian submanifolds $(L^+_j)$ and $(L_i^-)$ may not be mutually disjoint.
We will sometimes substitute $[0,1]$ in Definition~\ref{def:cob} by a more general interval $[R_-,R_+]$ for some real numbers $R_-< R_+$. 
In this case we will still denote by $V$ the cobordism inside $[R_-,R_+]\times \mathbb{R}\times M$.
\begin{definition}
 A Lagrangian cobordism $(V,(L_i^-),(L^+_j))$ is called monotone, if $V \subset [0,1]\times \mathbb{R} \times M$ is a
monotone Lagrangian submanifold.
\end{definition}
From now on we assume Lagrangian cobordisms to be monotone and connected.
It is sometimes useful to extend the definition of Lagrangian cobordism to Lagrangian submanifolds with cylindrical ends.
For this we consider the $\mathbb{R}$-extension $\tilde{V}\subset \tilde{M}$ of a 
Lagrangian cobordism $(V,\coprod_i L_i, \coprod_j L'_j)$, which is given by
\begin{equation*}
 \tilde{V}^{n+1}:= (\coprod_i (-\infty, 0]\times \{a_i^-\}\times L_i) \cup V \cup (\coprod_j [1, \infty)\times \{a_j^+\} \times L_j)\subset \mathbb{R}^2\times M.
\end{equation*}
We identify $\mathbb{C}\cong \mathbb{R}^2$ and fix the standard complex structure $i$ on $\mathbb{R}^2$.
\begin{definition}
Let $\{ L_j^+ \}_{1 \leq j \leq r_+}$ and $\{ L_i^- \}_{1 \leq i \leq r_-}$ be two families of closed Lagrangian submanifolds of $M$.
A Lagrangian submanifold with cylindrical ends, is a Lagrangian submanifold 
$\tilde{V}^{n+1}\subset \tilde{M}^{2n+2}$ without boundary, such that
\begin{enumerate}
 \item For every $a<b$, the subset $\tilde{V}|_{[a,b]\times \mathbb{R}}$ is compact.
 \item There exist constants $R_-$ and $R_+$, with $R_- \leq R_+$, as well as integers $\{a^-_1, \cdots, a^-_{r_-}\}$ and $\{a^+_1, \cdots, a^+_{r_+}\}$, such that 
\begin{equation*}
 \begin{array}{ccl}
  \tilde{V}|_{(-\infty, R_-] \times \mathbb{R}} & = & \coprod_{i=1}^{r_-}(-\infty, R_-] \times \{ a^-_i\} \times L^-_i\\
  \tilde{V}|_{[R_+, \infty) \times \mathbb{R}} & = & \coprod_{j=1}^{r_+}[R_+,+\infty) \times \{ a^+_j\} \times L^+_j.\\
 \end{array}
\end{equation*}
We call $\tilde{V}|_{[R_-,R_+]\times \mathbb{R}}$ the compact part of $\tilde{V}$.
\end{enumerate}
Let $R \leq R_-$. We call the sets $E_R^-(\tilde{V}):=\tilde{V}|_{(-\infty, R] \times \mathbb{R}}$ the negative cylindrical end of $\tilde{V}$, and similarly for $R \geq R_+$ we call $E_R^+(\tilde{V}):=\tilde{V}|_{[R, \infty) \times \mathbb{R}}$ the positive cylindrical end of $\tilde{V}$.
\end{definition}
A Lagrangian submanifold with cylindrical ends is said to be cylindrical outside of a compact set $K\subset \mathbb{R}^2$ if $V|_{\mathbb{R}^2 \setminus K}$ is of the form $ \coprod_{i=1}^{r_-}(-\infty, R_-] \times \{ a^-_i \} \times L^-_i \cup \coprod_{j=1}^{r_+}[R_+,+\infty) \times \{a^+_j \} \times L^+_j$.
The ends of the form $ \coprod_{i=1}^{r_-}(-\infty, R_-] \times \{ a^-_i \} \times L^-_i \cup \coprod_{j=1}^{r_+}[R_+,-\infty) \times \{ a^+_j \} \times L^+_j$ are called horizontal ends.

\begin{remark}~\label{rmk:cob_cylindrical}
Since the two notions of $V$ and $\tilde{V}$ are closely related we sometimes do not distinguish between them.
\end{remark}

\section{The Quantum Homology for Lagrangian Cobordism}\label{sec:LQHforLC}

Applying Lemma~\ref{lem:curve_openmapping} and~\ref{automatictransversality} we extend the arguments of the approach in~\cite{QuantumStructures} to give a generalization to our particular setting.
Let $\tilde{V} \subset (\tilde{M}, \omega_{\mathbb{R}^2}\oplus \omega)$ be a monotone Lagrangian submanifold, which is cylindrical outside of $K\times M$, for $K\subset \mathbb{R}^2$ compact, and such that the compact part $V:=\tilde{V}|_{[R_-,R_+]\times \mathbb{R}}$ is a cobordism $(V, \{L_i^-\}_{i=1}^{r_-}, \{L_j^+\}_{j=1}^{r_+})$.
Consider a subset $I_- \subset \{1, \cdots, r_- \}$ and $J_+ \subset \{ 1, \cdots, r_+\}$. 
Let $S$ be the union
\begin{equation*}
 S= (\coprod_{i \in I_-} \{(R_-,a_i^-)\}\times L_i^-) \cup (\coprod_{j\in J_+} \{(R_+,a_j^+)\}\times L_j^+)\subset \partial V
\end{equation*}
of boundary components of $V$.
We let $V^{\epsilon}:=\tilde{V}|_{[R_- -\epsilon, R_+ +\epsilon]\times \mathbb{R}}$ denote a small extension of the compact part $V$ of $\tilde{V}$.
Here $R_-$ and $R_+$ are such that $\tilde{V}$ is cylindrical outside of $[R_-, R_+]\times\mathbb{R}$.
Similarly, $\tilde{M}^{\epsilon}:=\tilde{M}|_{[R_- -\epsilon, R_+ +\epsilon]\times \mathbb{R}}$, such that $V^{\epsilon}\subset \tilde{M}^{\epsilon}$ is a Lagrangian submanifold.
This is a compact subset of $\tilde{V}$ and a compact manifold with boundary.
Denote by $\tilde{\mathcal{J}}$ the set of all $\omega$-compatible almost complex structures on $\tilde{M}$.
As before, we choose an almost complex structure $\tilde{J}\subset \tilde{\mathcal{J}}$ on $\tilde{M}$, such that the projection $\pi: \tilde{M}\rightarrow \mathbb{R}^2$ is holomorphic outside of the set $K\times M$. 
We may assume that $K\subset [R_-, R_+]\times \mathbb{R}$ is compact and we denote the restriction of $\tilde{J}$ to $\tilde{M}^{\epsilon}$ by $\tilde{J}$, as before.
Let $S^{\epsilon}$ be the union of the connected components of the boundary of $V^{\epsilon}$ that correspond to $S$, i.e.
\begin{equation*}
 S^{\epsilon}=(\coprod_{i \in I_-} \{ (R_-- \epsilon,a_i^-)\}\times L_i^-) \cup (\coprod_{j\in J_+} \{ ( R_+ +\epsilon,a_j^+)\}\times L_j^+) \subset \partial V^{\epsilon}.
\end{equation*}
Let $\tilde{f}:V^{\epsilon}\rightarrow \mathbb{R}$ be a Morse function and $\tilde{\rho}$ a Riemannian metric on $V^{\epsilon}$, such that the pair $(\tilde{f}, \tilde{\rho})$ is Morse-Smale.
Moreover, we choose $\tilde{f}$ such that the negative gradient $-\nabla \tilde{f}$ points outside of $\partial V^{\epsilon}$ along $S^{\epsilon}$ and inside of $\partial   V^{\epsilon}$ along $\partial V^{\epsilon} \setminus S^{\epsilon}$.
In particular $-\nabla \tilde{f}$ is transverse to $\partial V^{\epsilon}$.
We call such a function $\tilde{f}$ a \emph{Morse function respecting the exit region $S^{\epsilon}$}.

Let $V^{\epsilon} \subset (\tilde{M}^{\epsilon},\omega)$ be as above.
Assume that $V^{\epsilon}$ is monotone with minimal Maslov number $N_V^{\epsilon} \geq 2$.
Let $\tilde{J}$ be an almost complex structure on $\tilde{M}^{\epsilon}$ as before.
Denote by $D$ the unit disk in $\mathbb{C}$.
To define the pearl complex for Lagrangian cobordism we introduce the following moduli spaces, which are analogous to the ones that were defined for the compact setting in~\cite{LQH} and~\cite{QuantumStructures}.
\begin{definition}
Let $x$ and $y$ be two points in $V^{\epsilon}$. Denote by $\phi$ the flow of the negative gradient $-\nabla \tilde{f}$ of $\tilde{f}$.
Let $l \geq 1$ and let $\lambda$ be a non-zero class in $H^D_2(\tilde{M}^{\epsilon},V^{\epsilon})\subset H_2(\tilde{M}^{\epsilon},V^{\epsilon})$, where $H^D_2(\tilde{M}^{\epsilon}, V^{\epsilon})$ denotes the image of $\pi_2(\tilde{M}^{\epsilon},V^{\epsilon})$ under the Hurewicz homomorphism.
Consider the space of all sequences $(u_1, \dots u_l)$ that satisfy:
\begin{itemize}
  \item[(i)] For every $i$, $u_i:(D, \partial D) \rightarrow (\tilde{M}^{\epsilon},V^{\epsilon})$is a non-constant $\tilde{J}$-holomorphic disk.
  \item[(ii)] There exists $-\infty \leq t' < 0$ such that $\phi_{t'}(u_1(-1))=x$.
  \item[(iii)] There exists $0 < t'' \leq \infty$ such that $\phi_{t''}(u_l(1))=y$.
  \item[(iv)] For every $i$ there exists $0 < t_i < \infty$ such that $\phi_{t_i}(u_i(1))=u_{i+1}(-1)$.
  \item[(v)] The sum of the classes of the $u_i$ equals $\lambda$, i.e. $\sum\limits_{i=1}^{l} [u_i]=\lambda \in H^D_2(\tilde{M}^{\epsilon},V^{\epsilon})$.
\end{itemize}
Let $\mathcal{P}_{prl}(x,y,\lambda;\tilde{f},\tilde{\rho},\tilde{J})$ be the space of all such sequences, modulo the following equivalence relation.
Two elements $(u_1, \dots u_l)$ and $(u'_1, \dots u'_k)$ are equivalent if $l=k$ and for every $i$ there exists an automorphism $\sigma_i\in Aut(D)$ with $\sigma_i(-1)=-1$, $\sigma_i(1)=1$ and $u_i'=u_i\circ \sigma$.
We call elements in $\mathcal{P}_{prl}(x,y,\lambda;\tilde{f},\tilde{\rho},\tilde{J})$ pearly trajectories connecting $x$ and $y$ of class $\lambda$.
If $\lambda=0$ set $\mathcal{P}_{prl}(x,y,0;\tilde{f},\tilde{\rho},\tilde{J})$ to be the space of parametrized trajectories of $\phi$ connecting $x$ and $y$.
We sometimes denote $\mathcal{D}:=(\tilde{f},\tilde{\rho};\tilde{J})$.
\end{definition}
\begin{remark}
If $x$ and $y$ are critical points of $\tilde{f}$ then $t'=-\infty$, $t''=\infty$, $u_1(-1) \in W^u(x)$ and $u_l(1) \in W^s(y)$, where $W^u(x)$ and $W^s(y)$ are the unstable and stable manifold of $x$ and $y$ respectively. 
If not stated otherwise, we assume from now one that $x$ and $y$ are critical points.
\end{remark}
The virtual dimension of $\mathcal{P}_{prl}(x,y,\lambda;\tilde{f},\tilde{\rho},\tilde{J})$ is 
$$\delta_{prl}(x,y,\lambda):= |x|-|y|+\mu(\lambda)-1.$$
Lemma~\ref{lem:curve_openmapping} and~\ref{automatictransversality} together with the methods in~\cite{QuantumStructures} ensure that if $\delta_{prl}(x,y,\lambda)\leq1$, the space $\mathcal{P}_{prl}(x,y,\lambda;\tilde{f},\tilde{\rho},\tilde{J})$ is a smooth manifold of the dimension equal to its virtual dimension for a generic triple $(\tilde{f},\tilde{\rho},\tilde{J})$.
Put $C_i((V^{\epsilon},S^{\epsilon});\tilde{f},\tilde{\rho},\tilde{J}):=(\langle Crit(\tilde{f}) \rangle \otimes \Lambda)_i$
and by abuse of notation write 
$$C_i((V,S);\tilde{f},\tilde{\rho},\tilde{J}):=C_i((V^{\epsilon},S^{\epsilon});\tilde{f},\tilde{\rho},\tilde{J})=(\langle Crit(\tilde{f}) \rangle \otimes \Lambda)_i.$$
To define a differential on $C_i((V^{\epsilon},S^{\epsilon});\tilde{f},\tilde{\rho},\tilde{J})$ requires that we orient the moduli spaces.
If we assume the Lagrangian cobordism $V$ to be spin, the choice of a spin structure induces an orientation on the moduli spaces of pseudo holomorphic disks with boundary in $V$.
Choose an orientation of the stable and unstable submanifolds of a Morse function. 
The moduli spaces of pearly trajectories can then be written as a fiber product of these two types of oriented manifolds, and hence is itself an oriented manifold.
More details can be found in~\cite{LagrTop}.
To abbreviate the notation we write $\overline{\mu}$ for $\frac{\mu}{N_V}$.
Now we define the differential for $x\in Crit(\tilde{f})$ by
\begin{equation}\label{def:differential_Z}
 d(x):=\smashoperator{ \sum\limits_{\begin{subarray}{c} y\in Crit(\tilde{f})\\ \delta_{prl}(x,y,0)=0\end{subarray}}} 
\sharp \mathcal{P}_{prl}(x,y,0;\mathcal{D})y + 
\smashoperator{ \sum\limits_{\begin{subarray}{c} y\in Crit(\tilde{f}), \lambda\neq0\\ \delta_{prl}(x,y,\lambda)=0 \end{subarray}}}
(-1)^{|y|} \sharp \mathcal{P}_{prl}(x,y,\lambda;\mathcal{D})y 
t^{\overline{\mu}(\lambda)},
\end{equation}
and extend it linearly over $C(V^{\epsilon},S^{\epsilon};\tilde{f},\tilde{\rho},\tilde{J})$.
Now, if $V$ is not spin, then we can work over a ground ring $\mathcal{R}$ with characteristic two and we define
\begin{equation}\label{def:differential}
 d(x):=\smashoperator{\sum\limits_{\begin{subarray}{c} y,\lambda \\  \delta_{prl}(x,y,\lambda)=0 \end{subarray}}}
\sharp_2 \mathcal{P}_{prl} 
(x,y,\lambda;\mathcal{D}) t^{\overline{\mu}(\lambda)}.
\end{equation}

The next proposition proves the first part of Theorem~\ref{thm:QuantumStructures_LQHforLC}.

\begin{proposition}\label{prop:LQHforLC}
If $\mathcal{D}_S:=(\tilde{f},\tilde{\rho},\tilde{J})$ is generic, then the pearl complex
\begin{equation}
 C((V^{\epsilon},S^{\epsilon});\mathcal{D}_S):=(\langle Crit(\tilde{f}) \rangle \otimes 
\Lambda,d)
\end{equation}
is a well-defined chain complex and its homology is independent of the choices 
of $(\tilde{f},\tilde{\rho},\tilde{J})$.
It will be denoted by $QH_*(V,S)$.
\end{proposition}

The proof of this proposition goes along the lines of the proof of the analogous results in~\cite{QuantumStructures}. 
As already mentioned in Section~\ref{sec:OutlineofProof} we need the following lemma.
\begin{lemma}\label{lem:traj_awayfromboundary}
 Let $\mathbf{u}\in \mathcal{P}_{prl}(x,y,\lambda;\tilde{f},\tilde{\rho},\tilde{J})$ be a pearly trajectory connecting two critical points of $\tilde{f}$.
Then $\mathbf{u}$ does not reach the boundary $\partial V^{\epsilon}$.
\end{lemma}
\begin{proof}
Recall that by definition $-\nabla \tilde{f}$ is transverse to the boundary $\partial V^{\epsilon}$. 
In particular, there are no critical points of $\tilde{f}$ on $\partial V^{\epsilon}$.
Thus, non of the trajectories $\phi_t(u_i(1))$, $i=1, \dots l$ and $\phi_t(x)$ can reach the boundary.
If one of the $\tilde{J}$-holomorphic disks touches the boundary then by Lemma~\ref{lem:curve_openmapping} $\pi \circ u$ is constant. 
Hence the trajectories ending and starting at $u_i(-1)$ and $u_i(1)$ would reach $\partial V^{\epsilon}$, which is a contradiction.
\end{proof}
Now we explain the additional steps and ideas needed to generalize the arguments to give a proof of Proposition~\ref{prop:LQHforLC}.
\begin{proof}[Proof of Proposition~\ref{prop:LQHforLC}]
The proof consists of three steps.
The first step is covered by Lemma~\ref{lem:curve_openmapping}, ~\ref{automatictransversality} and~\ref{lem:traj_awayfromboundary}.
More precisely, Lemma~\ref{lem:curve_openmapping} implies Lemma~\ref{lem:traj_awayfromboundary}. 
Lemma~\ref{automatictransversality} ensures that we can use the same methods as in~\cite{QuantumStructures} to show that the moduli spaces of pearly trajectories of virtual dimension at most $1$ form smooth manifolds and that they are compact if their dimension is zero.

The second part of the proof amounts to show that $d\circ d=0$.
Let $\mathcal{P}$ be a one dimensional moduli space of pearly trajectories.
We will argue that the boundary of the compactification $\partial \overline{\mathcal{P}}$ has the same description as for the compact case.
Recall that for the compact setting in~\cite{QuantumStructures} the boundary of a one dimensional moduli space of pearly trajectories corresponds to one of the following three options,
\begin{enumerate}\label{enum:boundary}
 \item[(1)] One of the flow lines of the negative gradient of the Morse function breaks at a critical point.
 \item[(2)] One of the flow lines of the negative gradient of the Morse functions contracts to a constant point.
 \item[(3)] Bubbling of a pseudo-holomorphic disc of the pearly trajectory.
\end{enumerate}
Since the Lagrangian $V^{\epsilon}$ has a boundary, it could be that a pearly trajectory in $\mathcal{P}$ converges to something, which is not of the form given in the description above. 
More specific, what could happen is that a flow line of $-\nabla \tilde{f}$ breaks into two flow lines, both crossing the boundary $\partial V^{\epsilon}$.
But since the negative gradient $-\nabla \tilde{f}$ is transverse to each boundary component, such a breaking cannot exist.
Therefore we have, as in the compact setting, that
\begin{equation*}
 \partial\overline{\mathcal{P}}=\{ \text{elements of type (1)} \} \cup \{ \text{elements of type (2)} \}\cup \{ \text{elements of type (3)} \}.
\end{equation*}
The rest of the proof is analogous to the compact case.

As a last step it is left to show that the definition of the quantum homology is independent of the choice of the data $\mathcal{D_S}=(\tilde{f},\tilde{\rho},\tilde{J})$. 
Recall that we need a smooth family of almost complex structures $\tilde{J}_t=:\mathbb{J}$ on $V^{\epsilon}$, that all fulfill the properties needed for the definition of $QH(V,S)$.
For every $t$, we may choose $\tilde{J}_t=i\oplus J$ outside of $K\times M$, where $J$ is an almost complex structure on $M$.
Then the homotopy $\tilde{J}_t$ is constant equal to $i\oplus J$ outside of $K\times M$.
Clearly we have an analogous result of Theorem~\ref{lem:traj_awayfromboundary} for the comparison moduli spaces $\mathcal{P}_{comp}$.
Hence, the rest of the proof follows as in~\cite{QuantumStructures}.
\end{proof}
\begin{remark}
The extension $V^{\epsilon}$ and its boundary $S^{\epsilon}$, as well as $\tilde{M}^{\epsilon}$ are useful for the definition of the quantum homology $QH(V,S)$.
However, if there is no ambiguity we will from now on write $V$, $\tilde{M}$ and $S$ for 
the extensions $V^{\epsilon}$, $\tilde{M}^{\epsilon}$ and $S^{\epsilon}$.
\end{remark}

\section{The Augmentation}\label{sec:augmentation}
The aim of this section is to describe an augmentation map of the quantum homology.
Most of the results are generalizations of the results given in~\cite{QuantumStructures}.

The following lemma gives a specific Morse function on $V$, which is useful for the construction of the augmentation map.
\begin{lemma}\label{lem:singlemin}
There exists a Morse function $\tilde{f}$ on $V$, with $-\nabla \tilde{f}$ pointing inside along the boundary $\partial V$ and with a single minimum, which we call $x_0$.
\end{lemma}
\begin{proof}
In~\cite{Hirsch} Hirsch gives a proof of the existence of a Morse function with a unique minimum on a closed manifold.
The methods from~\cite{Hirsch} work also in this more general setting, since the idea of the proof is a stepwise elimination of the critical points of minimal index until there in only one left.
This is done by adjusting the Morse function such that three critical points of index $0$, $0$ and $1$ are locally replaced by one critical point of index $0$.
It is easy to see that this can be done away from the boundary if $-\nabla \tilde{f}$ is transverse to it. 
Moreover, since $-\nabla \tilde{f}$ points inside along the whole boundary this procedure gives a Morse function on $V$ with a single minimum $x_0$.
For more details on these methods the reader is referred to~\cite{Hirsch}.
\end{proof}

\begin{proposition}\label{prop:min_notboundary}
 Let $\tilde{f}$ be a Morse function on $V$ with $-\nabla \tilde{f}$ pointing inside along the
boundary $\partial V$ and with a single minimum, say $x_0$.
Then the minimum $x_0$ is not a boundary of the pearl complex of $V$.
\end{proposition}

\begin{proof}
It suffices to consider zero dimensional moduli spaces $\mathcal{P}_{prl} (x,x_0,\lambda)$.
Suppose first that $\mu(\lambda)\neq 0$ and let $x\in Crit(\tilde{f})\setminus \{ x_0 \}$.
Then we compute 
$$dim \mathcal{P}_{prl}(x, x_0, \lambda)=|x|-|x_0|+\mu(\lambda)-1\geq 1-0+2-1=2>0.$$
Hence, if $\lambda \neq 0$ then $\mathcal{P}_{prl}(x, x_0, \lambda)$ is not zero dimensional. 
Therefore we may assume that $\lambda = 0$ and in particular the index of $x$ is $1$. 
As a consequence the pearl differential is equivalent to the Morse differential in this case.
But the Morse homology of $V$ is non trivial in degree zero and $x_0$ is the only critical point of index zero. 
We conclude that $x_0$ is not a boundary.
\end{proof}
\begin{proposition}\label{prop:augm}
There exists a canonical, degree preserving augmentation 
\begin{equation*}
 \epsilon_{V}:QH_*(V) \rightarrow \Lambda,
\end{equation*}
which is a $\Lambda$-module map.
\end{proposition}
\begin{proof}
For a choice of data $\mathcal{D}=(\tilde{f},\tilde{\rho},\tilde{J})$ we define $\epsilon_{V,\mathcal{D}}$ on the chain level as follows.
\begin{equation}\label{eq:augm}
 \begin{array}{cccc}
  \epsilon_{V,\mathcal{D}}: & Crit(\tilde{f}) & \rightarrow & \Lambda\\
              & x 	&\mapsto & \begin{cases}
				  1 \text{, if }  |x| =0\\
				0 \text{, else }\\
                  	         \end{cases},
 \end{array}
\end{equation}
and extend it linearly over $C((V,S);\mathcal{D})$.
Here, the differential on $\Lambda$ is trivial. 
Assume now that $\tilde{f}$ is a Morse function as in Lemma~\ref{lem:singlemin}.
Then, $\epsilon_{V,\mathcal{D}}$ is a chain map, as the following calculation shows.
Clearly $\epsilon_{V,\mathcal{D}}(d(x_0))=0=d(\epsilon_{V,\mathcal{D}}(x_0))$.
If $x\neq x_0$ then $d\circ \epsilon_{V,\mathcal{D}}(x)=0$ and
\begin{equation*}
\begin{array}{ccl}
  \epsilon_{V,\mathcal{D}} \circ d(x) & = & \epsilon_{V,\mathcal{D}} \left( \sum\limits_{\begin{subarray}{c} y\\
\delta_{prl}=0\\
                                                                        \end{subarray}} 
\sharp \mathcal{P}_{prl}(x,y,0)y + 
\sum\limits_{\begin{subarray}{c} y,\lambda\neq0\\
\delta_{prl}=0\\
                                                                        \end{subarray}}
(-1)^{|y|} \sharp \mathcal{P}_{prl}(x,y,\lambda)y 
t^{\overline{\mu}(\lambda)} \right) \\
			& = & \epsilon_{V,\mathcal{D}} \left(  \sum\limits_{\begin{subarray}{c}
\delta_{prl}=0\\
                                                                        \end{subarray}} 
\sharp \mathcal{P}_{prl}(x,x_0,0)x_0 + 
\sum\limits_{\begin{subarray}{c} \lambda\neq0\\
                                                                         
\delta_{prl}=0\\
                                                                        \end{subarray}}
(-1)^{|x_0|} \sharp \mathcal{P}_{prl}(x,x_0,\lambda)x_0 
t^{\overline{\mu}(\lambda)} \right) \\
			& = & 0,\\
\end{array}
\end{equation*}
where the second equality holds since $\tilde{f}$ has the single minimum $x_0$ and 
$\epsilon_{V,\mathcal{D}}(y)=0$ whenever $y\neq x_0$.
The rest follows from Proposition~\ref{prop:min_notboundary}.

Let $\mathcal{D'}$ be another generic data.
To show that the map is canonical we need to check that $\epsilon_{V,\mathcal{D'}}\circ \phi^{\mathcal{D'}}_{\mathcal{D}} = \epsilon_{V,\mathcal{D}} $,
where $\phi^{\mathcal{D'}}_{\mathcal{D}}$ is the comparison isomorphism between the chain complex.
A dimension calculation shows that the comparison moduli space $\mathcal{P}_{comp}(x,y',\lambda)$, with $|y'|=0$ has dimension zero if and only if $|x|=-\mu(\lambda)=0$. 
In other words, there are no non-empty moduli spaces of dimension zero connecting a critical point $x\neq x_0$ of $\tilde{f}$ to a critical point $y'$ of $\tilde{f}'$ of index zero.
Thus, for $x\neq x_0$ we see that 
$$\epsilon_{V,\mathcal{D'}}\circ \phi^{\mathcal{D'}}_{\mathcal{D}}(x)=0=\epsilon_{V,\mathcal{D}}(x).$$
Consider the case that $x=x_0$. 
By definition, $\epsilon_{V,\mathcal{D'}}(y')=0$ for every $y'$ with $|y'|\neq0$. 
If $|y'|=0$, then $\mu(\lambda)$ must be zero by the same argument as above. 
We compute
\begin{equation}
\begin{array}{ccc}
  \epsilon_{V,\mathcal{D'}}\circ \phi^{\mathcal{D'}}_{\mathcal{D}}(x_0) &=&
\epsilon_{V,\mathcal{D'}} \left( \sum\limits_{\begin{subarray}{c} |y'|=0\\
\delta_{comp}=0 \end{subarray}} \sharp \mathcal{P}_{comp}(x_0,y',0)y'\right)\\
 &=& \sum\limits_{\begin{subarray}{c} |y'|=0\\
\delta_{comp}=0 \end{subarray}} \sharp \mathcal{P}_{comp}(x_0,y',0).\\
\end{array}
\end{equation}
Recall that the moduli spaces were defined using a Morse cobordism $(F,G)$ between $(\tilde{f},\tilde{\rho})$ and $(\tilde{f}',\tilde{\rho}')$. 
Notice that $\sum\limits_{\begin{subarray}{c}|y'|=0 \\ \delta_{comp}=0 \end{subarray}} \sharp \mathcal{P}_{comp}(x_0,y',0)y'$ 
is the Morse part of the differential for the Morse function $H$ of the critical point $x_0$ of $H$.
As a critical point of $H$, $x_0$ has index one and it lies on the boundary.
Standard Morse theoretic arguments show that the Morse differential of $x_0$ must be one critical point of $H$ of index zero. 
In other words, $\sum\limits_{\begin{subarray}{c} |y'|=0\\ \delta_{comp}=0 \end{subarray}} \sharp\mathcal{P}(x_0,y',0)=1$, which proves the statement.
\end{proof}

\section{The Quantum Product}\label{sec:QuantumProduct}

In this section we define a product on the quantum homology.
\begin{proposition}\label{prop:productandunit}
 There exists a $\Lambda$-bilinear map
\begin{equation*}
\begin{array}{ccc}
  *:QH_i(V,S)\otimes QH_j(V,S) & \rightarrow & QH_{i+j-(n+1)}(V,S) \\
  \alpha \otimes \beta & \mapsto & \alpha * \beta,\\
\end{array}
\end{equation*}
which endows $QH_*(V,S)$ with the structure of a (possibly non-unital) ring.
Moreover, if $S=\partial V$, then the product turns $QH_*(V,\partial V)$ into a ring with a unit.
\end{proposition}
We start by defining the necessary moduli spaces.
\begin{definition}\label{def:P_prod}
Fix three Morse functions $\tilde{f}$, $\tilde{f}'$ and $\tilde{f}''$ on $V$ respecting the exit region $S$, a Riemannian metric $\tilde{\rho}$ and a generic almost complex structure $\tilde{J}\in \tilde{\mathcal{J}}$ with the properties described in Section~\ref{sec:LQHforLC}.
Let $x$, $y$ and $z$ be critical points in $L$ of $\tilde{f}$, $\tilde{f}'$ and $\tilde{f}''$, respectively. 
Let $\lambda$ be a non-zero class in $H^D_2(\tilde{M},V)\subset H_2(\tilde{M},V)$, where $H^D_2(\tilde{M},V)$ denotes the image of $\pi_2(\tilde{M},V)$ under the Hurewicz homomorphism.
Consider the space of all tuples $(\mathbf{u}, \mathbf{u'}, \mathbf{u''}, v)$ that satisfy:
\begin{itemize}
  \item[(i)] $v:(D, \partial D) \rightarrow (\tilde{M},V)$ is a 
$\tilde{J}$-holomorphic disk, possibly constant.
$v$ is also called the core of $(\mathbf{u}, \mathbf{u'}, \mathbf{u''}, v)$.
  \item[(ii)] Denote $\tilde{x}=v(e^{-2\pi i/3})$, $\tilde{y}=v(e^{2\pi i/3})$ and $\tilde{z}=v(1)$. The points $\tilde{x}$, $\tilde{y}$ and $\tilde{z}$ are no critical points. Then:
	      $\mathbf{u}\in 
\mathcal{P}_{prl}(x,\tilde{x},\mathbf{\xi};\tilde{J},\tilde{\rho},\tilde{f})$,
	      $\mathbf{u'}\in 
\mathcal{P}_{prl}(y,\tilde{y},\mathbf{\xi'};\tilde{J},\tilde{\rho},\tilde{f}')$,
	      $\mathbf{u''}\in 
\mathcal{P}_{prl}(\tilde{z},z,\mathbf{\xi''};\tilde{J},\tilde{\rho},\tilde{f}'')$,
      for some $\mathbf{\xi},\mathbf{\xi'}, \mathbf{\xi''} \in H^D_2(\tilde{M},V)$.
  \item[(iii)] $\mathbf{\xi}+ \mathbf{\xi'}+ \mathbf{\xi''}+[v]=\mathbf{\lambda}$.
\end{itemize}
Denote by $\mathcal{P}_{prod}(x,y,z,\lambda;\tilde{f},\tilde{f}',\tilde{f}'',\tilde{\rho},\tilde{J})$ the space of all such sequences.
An element in this space is called a figure-$Y$ pearly trajectory from $x$ and $y$ to $z$.
\end{definition}
We define the virtual dimension of $\mathcal{P}_{prod}(x,y,z,\lambda;\tilde{f},\tilde{f}',\tilde{f}'',\tilde{\rho},\tilde{J})$ by 
$$\delta_{prod}(x,y,z,\mathbf{\lambda}):=|x|+|y|-|z|+\mu(\lambda)-(n+1).$$
As before $\mathcal{P}_{prod}(x,y,z,\lambda;\tilde{f},\tilde{f}',\tilde{f}'',\tilde{\rho},\tilde{J})$ can be oriented by writing it as the fiber product of oriented moduli spaces.
Assume that $(\tilde{f},\tilde{f}',\tilde{f}'',\tilde{\rho})$ are in general position.
If $\delta_{prod}(x,y,z,\mathbf{\lambda})\leq1$, the moduli spaces $\mathcal{P}_{prod}(x,y,z,\lambda)$ form smooth manifolds of dimension equal to their virtual dimension and they are compact if $\delta_{prod}(x,y,z,\mathbf{\lambda})=0$. 
This is a result of Lemma~\ref{automatictransversality} and ~\ref{lem:prod} below.
The quantum product on the chain level is defined by counting the elements in the zero dimensional moduli spaces of figure-$Y$ pearly trajectories.
More precisely, for every $x\in Crit(\tilde{f})$, $y \in Crit(\tilde{f}')$ and we define
\begin{equation}\label{eq:product}
 x * y := \sum\limits_{ \begin{subarray}{c}
                         z,  \lambda \\ \delta_{prod}=0\\
                         \end{subarray}
} \sharp \mathcal{P}_{prod}(x, y, z, 
\mathbf{\lambda};\tilde{f},\tilde{f}',\tilde{f}'',\tilde{\rho}, \tilde{J}) z
t^{\overline{\mu}(\lambda)},
\end{equation}
where the sum runs over all $z\in Crit(\tilde{f}'')$ and $\mathbf{\lambda} \in H^D_2(\tilde{M},V)$, such that $\delta_{prod}(x,y,z,\mathbf{\lambda})=0$.
With this definition we have the following results.
\begin{lemma}\label{lem:prod}
 For a generic choice of data, the operation in~(\ref{eq:product}) is well defined and it is a chain map.
It induces a product in homology, which is independent of the choices made in the construction.
\end{lemma}
\begin{lemma}\label{lem:assoc}
 The product
\begin{equation*}
 *:QH_i(V,S)\otimes QH_k(V,S)\rightarrow QH_{i+k-(n+1)}(V,S)
\end{equation*}
is associative.
\end{lemma}

At this point we give a more general version of Lemma~\ref{lem:traj_awayfromboundary}.
We will use the notations introduced in~\cite{QuantumStructures}, which describe the 
elements of the moduli spaces as configurations modeled over planar oriented trees.
The edges correspond to Morse functions on $V$ and the vertices (except) the entry and 
exit vertices correspond to $\tilde{J}$-holomorphic disks with their boundary in $V$.
The entry and exit points are critical points of the Morse functions.

\begin{lemma}~\label{lem:traj_awayfromboundary_general}
Let $\mathcal{P}$ be a moduli space modeled over planar trees, as described above.
Assume 
that all the edges corresponding to flow lines of Morse functions on
$V$ respecting the exit region $S\subset \partial V$.
Suppose that the almost complex structure $\tilde{J}$ is such that the projection $\pi$ is $\tilde{J}$-holomorphic
outside of $K\times M$.
Then, the elements of the moduli space $\mathcal{P}$ cannot reach the boundary $\partial V$.
In particular, if the virtual dimension of $\mathcal{P}$ is at most $1$, it is a smooth 
manifold of dimension equal to its virtual dimension and it is compact if $0$-dimensional.
Moreover, the compactification of a one dimensional such moduli space has the same 
description as the compactification
of the analogous moduli space as in the setting of a compact connected Lagrangian. 
\end{lemma}

\begin{proof}[Proof of Lemma~\ref{lem:traj_awayfromboundary_general}]
The proof is analogous to the proof of Lemma~\ref{lem:traj_awayfromboundary} and the part 
of the proof of~\ref{prop:LQHforLC}, where the boundary of the compactification 
$\overline{\mathcal{P}}$ is described.
\end{proof}

\begin{proof}[Proof of Lemma~\ref{lem:prod}]
It follows directly form Lemma~\ref{lem:traj_awayfromboundary_general} and the methods in~\cite{QuantumStructures}.
\end{proof}

\begin{proof}[Proof of Lemma~\ref{lem:assoc}]
By Lemma~\ref{lem:traj_awayfromboundary_general}, the proof is the same as in~\cite{QuantumStructures}.
\end{proof}

For the existence of a unit in the case that $S=\partial V$ we need that $-\nabla 
\tilde{f}$ is transverse to $\partial V$ and points outside at the boundary.
\begin{proposition}\label{lem:singlemax_rel}
There exists a Morse function on $V$ with a single maximum $m$ and such that
$-\nabla \tilde{f}$ points outside at $\partial V$.
\end{proposition}
\begin{proof}
 This follows directly from Proposition~\ref{lem:singlemin}.
For example we take $-\tilde{f}$ for a Morse function $\tilde{f}$ given by \ref{lem:singlemin}.
\end{proof}

\begin{lemma}\label{lem:unit_rel}
There exists a canonical element $e_{(V,\partial V)}\in QH_{n+1}(V,\partial V)$, which is 
a unit with respect to the quantum product, i.e. $e_{(V,\partial V)}*x=x$ for every $x\in 
QH_*(V,\partial V)$.
\end{lemma}

\begin{proof}
Let $\tilde{f}$ be a Morse function as in Lemma~\ref{lem:singlemax_rel}.
We want to show that the single maximum $m$ represents a unit in the quantum homology $QH_*(V,\partial V)$.
Any non-void moduli space $\mathcal{P}_{prod}(m,x,\lambda)$ with $\mu(\lambda)>0$ is of dimension bigger than $0$.
Indeed, $\mathcal{P}_{prod}(m,x,\lambda)$ has dimension 
$$|m|-|x|+\mu(\lambda)-1=(n+1)-|x|+\mu(\lambda)-1.$$
If $\mu(\lambda)>0$ then 
$$dim\mathcal{P}_{prod}(m,x,\lambda)\geq (n+1)-n+2-1=2.$$
Thus, to compute $d(m)$ it suffices to consider the moduli spaces $\mathcal{P}_{prod}(m,x,0)$. 
This corresponds to computing the Morse differential. 
For the Morse differential, $m$ clearly is a cycle, since $H_{n+1}(V,\partial V)=\mathcal{R}$.
We compute $m * y$ on the chain level, to show that $m$ represents the unit.
Choose $(\tilde{f}, \tilde{f}',\tilde{f}'',\tilde{\rho})$ in general position.
We may assume $\tilde{f}'=\tilde{f}''$.
We have
\begin{equation*}
\begin{array}{cccc}
  m * y & = & \sum\limits_{\begin{subarray}{c}
			z,\lambda\\
			\delta_{prod}=0\\
			 \end{subarray}
} \sharp \mathcal{P}_{prl}(m,y,z,\lambda)zt^{\overline{\mu}(\lambda)}.\\
\end{array}
\end{equation*}
Since $\tilde{f}'=\tilde{f}''$, we see that the figure-$Y$ pearly trajectory gives rise 
to a pearly trajectory from $y$ to $z$.
The identity 
$$0=\delta_{prod}(m,y,z,\lambda)=|m|+|y|-|z|+\mu(\lambda)-(n+1)$$ 
implies that 
$$0=|y|-|z|+\mu(\lambda),$$ 
since $|m|=n+1$.
Assume that $|y|\neq|z|$ and $\mu(\lambda)\neq 0$, then we have
\begin{equation*}
 \delta_{prl}(y,z,\lambda)=|y|-|z|+\mu(\lambda)-1<0,
\end{equation*}
which is a contradiction.
We conclude that $\mu(\lambda)=0$ and $|y|=|z|$, and thus $y=z$.
We may assume that $y=z$ is not a critical point of $\tilde{f}$.
Because $m$ is the only maximum of $\tilde{f}$, there exists a unique flow line of 
$-\nabla \tilde{f}$ starting in $m$ and going through $y=z$. 
Therefore
\begin{equation*}
\begin{array}{cccccc}
  m * y & = & \smashoperator{\sum\limits_{\begin{subarray}{c}
			z,\lambda\\
			\delta_{prod}=0\\
			 \end{subarray}}
} \sharp \mathcal{P}_{prod}(m,y,z,\lambda)zt^{\overline{\mu}(\lambda)}
	& = & \pm y.
\end{array}
\end{equation*}
Recall that the conventions for the moduli spaces of pearly trajectories are such that 
$\mathcal{P}_{prod}(m,y,y,0)$ is oriented as the unstable manifold of $m$. 
In particular the sign in front of $y$ in the above calculation is positive.
Hence 
\begin{equation*}
  m * y  = y.
\end{equation*}
This proves that the cycle $m$ represents the unit in homology.

It is left to show that the definition of the unit on the chain level is canonical.
We need to show that the corresponding comparison chain morphism preserves the homology 
class of $m$.
Suppose $\phi_{\mathcal{D}}^{\mathcal{D}'}$ is such a comparison morphism. 
Assume that the Morse function $\tilde{f}'$ belonging to $\mathcal{D}'$ is another Morse 
function with a single maximum $m'$.
The dimension of the comparison moduli space $\mathcal{P}_{comp}(m,y',\lambda)$, where 
$m$ is the single maximum of $\tilde{f}$ and $y'$ is any critical point of $\tilde{f}'$, 
is 
$$|m|-|y'|+\mu(\lambda)=(n+1)-|y'|+\mu(\lambda).$$
Thus, its dimension is only zero if $\mu(\lambda)=0$ and $|y'|=n+1$.
In other words, we have $\phi_{\mathcal{D}}^{\mathcal{D}'}(m)=m'$.
\end{proof}
Denote the unit of $QH_*(V,\partial V)$ by $e_{(V,\partial V)}:=[m]$.
\begin{remark}
Notice that it is crucial that $S$ is the whole boundary.
If $S\neq \partial V$ then some Morse trajectories ending at critical points of index $n$ 
(i.e. one less than the maximal index) do not come from $m$ but rather enter the manifold 
$V$ through $\partial V\setminus S$, where $-\nabla \tilde{f}$ 
points inside.
This means in general that $m$ is not a cycle in this complex.
It is then not true any more that there exists a unique flow line of $-\nabla \tilde{f}$ 
through the point $y$ coming from $m$.
Therefore the proof fails in the case $S\neq \partial V$.
\end{remark}

\subsection{The Graded Ring Structure}

Recall that by $(G,\cap)$ we denote the monoid $G=\{\text{subsets of } \pi_0( \partial 
V) \}$, where the operation $\cap$ is the intersection and $\partial V$ is the unit.
The aim of this section is to show that $\bigoplus_{S\in G} QH(V,S)$ admits the structure 
of a graded ring over $G$, namely that there exist a product on $\bigoplus_{S\in G} 
QH(V,S)$ with the property that $QH(V,S)*QH(V,S') \subset QH(V,S\cap S')$. 
For more information about rings graded over monoids see for example~\cite{monoid}.

\begin{definition}\label{def:P_prod_general}
Let $S$ and $S'\in G$ be two collections of connected components of $\partial V$.
 Define the moduli space $\mathcal{P}_{prod}^{S,S',S\cap 
S'}(x,y,z,\lambda;\tilde{f},\tilde{f}',\tilde{f}'',\tilde{\rho},\tilde{J})$ similar as in 
definition~\ref{def:P_prod} with the only difference that the functions 
$\tilde{f},\tilde{f}'$ and $\tilde{f}''$ are Morse functions respecting the exit regions 
$S$, $S'$ and $S\cap S'$ respectively.
Then we define for every $x\in Crit(\tilde{f})$, $y \in Crit(\tilde{f}')$,
\begin{equation}\label{eq:product_general}
 x * y := \sum\limits_{ \begin{subarray}{c}
                         z,  \lambda \\ \delta_{prod}=0\\
                         \end{subarray}
} \sharp \mathcal{P}^{S,S',S\cap S'}_{prod}(x, y, z, 
\mathbf{\lambda};\tilde{f},\tilde{f}',\tilde{f}'',\tilde{\rho}, \tilde{J}) z
t^{\overline{\mu}(\lambda)},
\end{equation}
where the sum runs over all $z\in Crit(\tilde{f}'')$ and $\mathbf{\lambda} \in 
H^D_2(\tilde{M},V)$, such that $\delta_{prod}(x,y,z,\mathbf{\lambda})=0$.
\end{definition}

With the above definition we get the following result.
The proof of the next proposition is postponed to the end of this section.

\begin{proposition}\label{prop:P_prod_general}
 There exists a bilinear map 
\begin{equation*}
\begin{array}{ccc}
 *_{S,S'} : QH_i(V,S)\otimes QH_j(V,S')& \rightarrow & QH_{i+j-(n+1)}(V, S\cap S'),
\end{array}
\end{equation*}
which extends the usual quantum product and turns $\bigoplus_{S\in G} QH(V,S)$ into a 
graded ring over $(G,\cap)$ with unit $e_{(V,\partial V)}$.
\end{proposition}

As in the definition of the usual product we can prove that~(\ref{eq:product_general}) is 
well-defined and on the homology level it is invariant of the choice of the data. 
To prove the rest of the proposition we introduce a special Morse function, which allows 
us to relate the quantum homologies of $V$ with respect to different boundary parts.

\begin{definition}\label{def:specialMF0}
Let $\tilde{f}: V^{\epsilon} \rightarrow \mathbb{R}$ be a Morse function on  
$V^{\epsilon}$, such that $-\nabla \tilde{f}$ is transverse to the boundary $\partial 
V^{\epsilon}$ and points outside along $\partial V^{\epsilon}$.
Moreover we want $\tilde{f}$ to fulfill the following properties.
\begin{equation*}
        \begin{array}{lll}
        \tilde{f}(t,a_j^+,p)=f_j^+(p)+\sigma_j^+(t) & 
\sigma_j^+:[R_+,R_+ +\epsilon]\rightarrow \mathbb{R}, & p\in M, j=1,\dots, k_+, \\
	\tilde{f}(t,a_i^-,p)=f_i^-(p)+\sigma_i^-(t) & \sigma_i^-:[R_--\epsilon,R_- ]\rightarrow
\mathbb{R}, & p\in M, i=1,\dots, k_-, \\
       \end{array},
\end{equation*}
where $f_j^+:L_j^+\rightarrow \mathbb{R}$ and $ f_i^-:L_i^- \rightarrow \mathbb{R}$ are 
Morse functions on $L_j^+$ and $L_i^-$ respectively.
\begin{enumerate}
  \item [(i)] $\sigma_j^+(t)$ is a decreasing linear function for $t \in [R_+ + 
\frac{3\epsilon}{4},R_+ + \epsilon]$.
Furthermore $\sigma_j^+(t)$ has a critical point $t_{j,1}^+:=R_+ + \frac{\epsilon}{2}$ of 
index $1$, and a critical point $t_{j,0}^+:=R_+ + \frac{\epsilon}{4}$ of index $0$.
  \item [(ii)] $\sigma_i^-(t)$ is an increasing linear function for $t \in [R_- 
-\epsilon, R_- -\frac{3\epsilon}{4}]$.
Furthermore $\sigma_i^-(t)$ has a critical point $t_{i,1}^-:=R_- - \frac{\epsilon}{2}$ of 
index $1$ and a critical point $t_{i,0}^-:=R_- -\frac{\epsilon}{4}$ of index $0$.
\end{enumerate}
Let $S$ be a collection of connected components of $\partial V$ and let $N(S)$ be 
the neighborhood of $S^{\epsilon}$ given by
\begin{equation*}
 N(S):=\coprod_{L_i^- \in S} [R_- -\epsilon,  R_- -3\epsilon/8]\times \{a^-_i\} \times 
L_i^-\coprod_{L_j^+\in S} [R_++3\epsilon/8,  R_+ +\epsilon]\times \{a^+_j\} \times 
L_j^+.
\end{equation*}
Then the set $U(S)$ is defined by $U(S):=V^{\epsilon}\setminus N(S)$.
\end{definition}

The functions $\sigma_j^+(t)$ and $\sigma_i^-(t)$ are illustrated in figure~\ref{sigmas}.
\begin{figure}
 \centering
\includegraphics{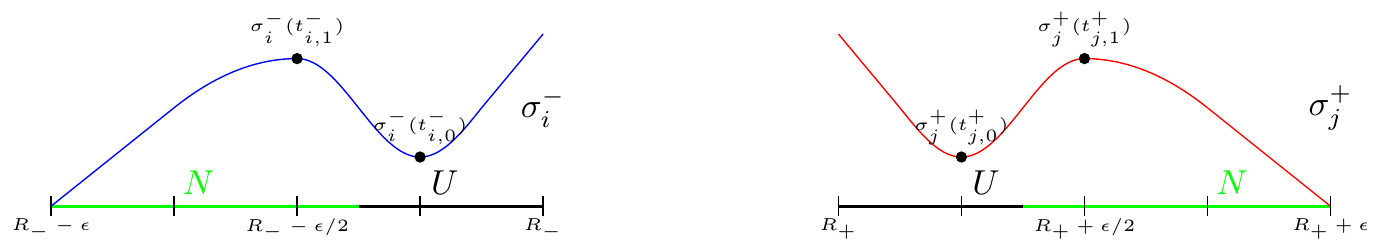}
\caption{The functions $\sigma_i^-$ and $\sigma_j^+$.}\label{sigmas}
 \end{figure}

Let $\tilde{f}$ be a Morse function as in Definition~\ref{def:specialMF0} respecting the 
exit region $\partial V$.
For a subset $S\subset \partial V$ denote by $f_{(V,S)}$ the Morse function given by 
restricting $\tilde{f}$ to the subset $U(\partial V \setminus S)$.
Notice that this is a Morse function respecting the exit region $S$.
Suppose that $S\subset S'$, then $C((V,S);\tilde{f}_{(V,S)},\rho,J)$ is naturally a 
subcomplex of $C((V,S');\tilde{f}_{(V,S')},\rho,J)$.
Let 
\begin{equation*}
 j^*_{S,S'}: QH(V,S)\rightarrow QH(V,S')
\end{equation*}
be the map, which is induced by the inclusion map $j_{S,S'}$ of the subcomplex 
$C((V,S);\tilde{f}_{(V,S)},\rho,J,\Lambda)$ into the chain complex 
$C((V,S');\tilde{f}_{(V,S')},\rho,J,\Lambda)$.
Fix three functions $\tilde{f},\tilde{f}'$ and $\tilde{f}''$ as in 
Definition~\ref{def:specialMF0}, which are in general position.
Since the definition of the product is independent of the choice 
of Morse functions, it suffices to prove the identities of the product using restrictions 
of the functions $\tilde{f},\tilde{f}'$ and $\tilde{f}''$ to appropriate subsets.
\begin{remark}\label{rmk:prodgeneral}
\hspace{2em}
\begin{enumerate}
 \item  Notice that Definition~\ref{def:P_prod_general} coincides on the chain level with 
the following definition.
If $a$ and $b$ are both elements of the subcomplex $C(V,S\cap 
S';\tilde{f}_{(V,S\cap 
S')},\rho,J)$, then $a*_{S,S'}b=a*b$, where $*$ denotes the usual product on $C(V,S\cap 
S';\tilde{f}_{(V,S\cap S')},\rho,J)$.
If either $a$ or $b$ is not contained in the subcomplex $C(V,S\cap 
S';\tilde{f}_{(V,S\cap 
S')},\rho,J)$, the following dimension calculation and the discussion above guarantee 
that in this case there exist no pearly trajectory from $a$ and $b$ to a critical point 
$c$, which lies in $C(V,S\cap S';\tilde{f}_{(V,S\cap S')},\rho,J)$.
Therefore, in this case we have $a*_{S,S'}b=0$.

\begin{proof}
We may assume without loss of generality that $a\in 
C(V,S;\tilde{f}_{(V,S)},\rho,J) \setminus C(V,S\cap S';\tilde{f}_{(V,S\cap 
S')},\rho,J)$.
Thus $a$ is of the form $(t_{j,1}^+,a_j^+,p)$ or $(t_{i,1}^-,a_i^-,p)$ for some 
critical point $p$ of $f$ and some $i$ or $j$, such that $L_j^+$ respectively $L_i^-$ 
belongs to $S\setminus (S\cap S')$.
Suppose there exist a trajectory from $a=(t_{i,1}^-,a_i^-,p)$ and $b$ to some element 
$c\in C(V,S\cap S';\tilde{f}_{(V,S\cap S')},\rho,J)$ such that the dimension of the 
moduli space containing it is zero. 
Then this implies the existence of a trajectory from $(t_{i,0}^-,a_i^-,p)$ and $b$ to 
$c$, which lies in a moduli space that has dimension one less (since 
$|(t_{i,0}^-,a_i^-,p)|=|(t_{i,1}^-,a_i^-,p)|-1$), i.e. it has dimension $-1$. 
This contradicts the assumption.
Hence we conclude $a*_{S,S'}b=0$.
\end{proof}

\item It is natural that for $a\in C(V,S;f_{(V,S)},\rho,J)$ and $b\in 
C(V,S';\tilde{f}_{(V,S')},\rho,J)$ the product $a*_{S,S'}b$ lies inside $C(V,S\cap 
S';\tilde{f}_{(V,S\cap S')},\rho,J)$.
Clearly, if $z\in C(V,\partial V ;f_{(V,\partial V)},\rho,J)\setminus C(V,S\cap 
S';\tilde{f}_{(V,S\cap S')},\rho,J)$, then there exist no pearly trajectory in 
$C(V,\partial V ;\tilde{f}_{(V,\partial V)},\rho,J)$ from $a$ and $b$ to $z$.
\end{enumerate}
\end{remark}

\begin{proof}[Proof of Proposition~\ref{prop:P_prod_general}]
By similar arguments as before, we know that~(\ref{eq:product_general}) is a 
chain map and induces a product on homology.
Using Morse functions $\tilde{f},\tilde{f}'$ and $\tilde{f}''$ as in 
Definition~\ref{eq:product_general} and suitable restrictions together with
Remark~\ref{rmk:prodgeneral} we can see that, if 
$S\subset T$ and $S'\subset T'$, then 
\begin{equation}\label{eq:prod_and_j}
 j^*_{S,T}(x)*_{T,T'}j^*_{S',T'}(y)= j^*_{S\cap S',T\cap T'}(x*_{S,S'}y).
\end{equation}
This identity implies that the product $*_{S,S'}$ extends the usual product given 
in Lemma~\ref{lem:prod} and that the unit of the ring $e_{(V,\partial V)} \in 
QH(V,\partial 
V)$ is a unit for the graded ring $\bigoplus_{S\in G} QH(V,S)$.
\end{proof}

\section{The Module Structure}\label{sec:ModuleStructure}

In this section we want to endow the quantum homology $QH(V,S)$ with the structure of a module over a version of the quantum homology related to the ambient manifold $\tilde{M}$.
In~\cite{QuantumStructures} the quantum homology of a closed Lagrangian has the structure of a two sided algebra over the quantum homology of the ambient manifold. 
In our case we are considering a Lagrangian cobordism.
There are two natural ways to define the ambient manifold.
Let $c$ be big enough such that $V:=\tilde{V}|_{[R_--\epsilon,R_+ +\epsilon]\times \mathbb{R}}=\tilde{V}|_{[R_--\epsilon,R_++\epsilon]\times[-c,c]}$ and we set 
$$T':=[R_--\epsilon,R_++\epsilon]\times (-c,c)\subset\mathbb{R}^2$$
 and 
$$R':=[R_--\epsilon,R_++\epsilon]\times [-c,c] \subset\mathbb{R}^2.$$
Then $T$ and $R$ denote the set obtained from $T'$ and $R'$ respectively by smoothening their boundaries and set $\tilde{M}_T:=T\times M$ and $\tilde{M}_R:=R\times M$.
The aim of this section is to prove the following result
\begin{proposition}
 There exists a bilinear map
\begin{equation*}
 \ast:QH_i(\tilde{M}_R,\partial \tilde{M_R}) \otimes QH_j(V,\partial V) 
\rightarrow 
QH_{i+j-(2n+2)}(V,\partial V),
\end{equation*}
which endows $QH_*(V,\partial V)$ with the structure of a two-sided algebra over the unital ring 
$QH_*(\tilde{M}_R,\partial \tilde{M_R})$.
\end{proposition}

\subsection{The Ambient Quantum Homologies \texorpdfstring{$QH_*(\tilde{M}_T,\partial\tilde{M}_T)$ and
$QH_*(\tilde{M}_R,\partial\tilde{M}_R)$}{QH and
QH}}\label{sec:ambient_T_R}

Before we start with the construction of the module structure, we give the definition of the ambient quantum homology and explain their structures and relations.
Choose the symplectic form $\omega_{\tilde{M}_T}=(\omega_{\mathbb{R}^2}\oplus \omega_M)|_T$ and $\omega_{\tilde{M}_R}=(\omega_{\mathbb{R}^2}\oplus \omega_M)|_R$ respectively.
Let $\tilde{J}$ be an $\omega$-compatible almost complex structure on $\tilde{M}_T$ and $\tilde{M}_R$ respectively and assume that the projection $\pi$ is $\tilde{J}$-holomorphic outside of some set $K\times M$, where $K\subset (R_-,R_+)\times (-c,c)\subset R \subset\mathbb{R}^2$ is compact, and such that $V$ is cylindrical outside of $K\times M$.

We can define the relative quantum homology of the pairs $(\tilde{M}_T,\partial \tilde{M}_T)$ and $(\tilde{M}_R,\partial \tilde{M}_R)$.
Additively the quantum homologies $QH_*(\tilde{M}_T,\partial \tilde{M}_T)$ and $QH_*(\tilde{M}_R,\partial \tilde{M}_R)$ are the same as the singular homologies $H_*(\tilde{M}_T,\partial \tilde{M}_T)$ and $H_*(\tilde{M}_R,\partial \tilde{M}_R)$.
In order to define a product structure on the quantum homology we need specific Morse functions on $\tilde{M}_T$ and $\tilde{M}_R$.
Let $\tilde{g}:\tilde{M}\rightarrow \mathbb{R}$ be a Morse function on $\tilde{M}$ such that $-\nabla \tilde{g}$ points outside along the boundary $\partial \tilde{M}_T$ and inside at $\partial \overline{\tilde{M}_T}\setminus \partial \tilde{M}_T$.
Similarly let $\tilde{h}:\tilde{M}\rightarrow \mathbb{R}$ be a Morse function on $\tilde{M}$ such that $-\nabla \tilde{h}$ points outside at the boundary $\partial \tilde{M}_R$.
For instance, we can take $\tilde{g}$ to be of the form $(\tau_T+ f)|_{\tilde{M}_T}$, where $f$ is a Morse function on $M$ and $\tau_T$ is a Morse function on $\mathbb{R}^2$ with a single critical point of index $1$ inside $T$, and such that $-\nabla \tau_T$ points outside of $\partial T$ and inside of $\partial\overline{T}\setminus\partial T$.
In the same way we could choose $\tilde{h}$ to be of the form $(\tau_R + f)|_{\tilde{M}_T}$, where $f$ is a Morse function on $M$ and $\tau_R$ is a Morse function on $\mathbb{R}^2$ with a single critical point of index $2$ inside $R$ such that $-\nabla \tau_R$ points outside along $\partial R$.

The quantum product can be defined by counting the elements of a moduli spaces modeled over trees with two entries and one exit point and one vertex of valence three corresponding to a $\tilde{J}$-holomorphic sphere.
The entry points correspond to two critical points $x$ and $y$ of two Morse functions in general position and the exit point is a critical point $z$ of a third Morse function in general position.
In the case of $\tilde{M}_T$ we use Morse functions $\tilde{g}$, $\tilde{g}'$ and $\tilde{g}''$, as they were defined above.
In the case of $\tilde{M}_R$ we use Morse functions $\tilde{h}$, $\tilde{h}'$ and $\tilde{h}''$.
The edges of the tree correspond to the flow lines of the corresponding Morse functions.
Compare this with Definition~\ref{def:P_prod}.
We denote these moduli spaces by $\mathcal{M}_{prod}(x,y,z,\lambda;\tilde{g},\tilde{g}',\tilde{g}'',\tilde{J},\tilde{\rho })$ and $\mathcal{M}_{prod}(x,y,z,\lambda;\tilde{h},\tilde{h}',\tilde{h}'', \tilde{J},\tilde{\rho})$.
Let $x\in Crit(\tilde{g})$ and $y\in Crit(\tilde{g}')$ two critical points of two Morse functions on $\tilde{M}_T$.
Then
\begin{equation}\label{eq:qprod_T}
 x * y:=\sum\limits_{z,\lambda} 
\sharp \mathcal{M}_{prod}(x,y,z,\lambda;\tilde{g},\tilde{g}',\tilde{g}'' ) 
z t^{\overline{\mu}(\lambda)},
\end{equation}
where the sum runs over all $z\in Crit(\tilde{g})$ and $\lambda$, such that 
$\mathcal{M}_{prod}(x,y,z,\lambda)$ is $0$-dimensional.
Similarly for $\tilde{M}_R$ and $x,y,z \in Crit(\tilde{h})$ we put
\begin{equation}\label{eq:qprod_R}
 x* y:=\sum\limits_{z,\lambda} 
\sharp \mathcal{M}_{prod}(x,y,z,\lambda;\tilde{h},\tilde{h}',\tilde{h}'' ) 
z t^{\overline{\mu}(\lambda)},
\end{equation}
with the sum taken over the zero dimensional moduli spaces.

At this point it is useful to extend the result of 
Lemma~\ref{lem:traj_awayfromboundary_general} to more general moduli spaces.
We use the conventions from~\cite{QuantumStructures} to describe the moduli 
spaces modeled over trees.
These trees are now allowed to be more general. More precisely, in addition to the moduli 
spaces modeled over trees from Lemma~\ref{lem:traj_awayfromboundary_general} we allow the 
following things
\begin{enumerate}
 \item The edges of the tree can also correspond to flow line of the negative gradient of 
a Morse function $\tilde{g}$ or $\tilde{h}$ on one of the ambient manifolds 
$\tilde{M}_T$ or $\tilde{M}_R$.
 \item The edges may correspond to $\tilde{J}$-holomorphic spheres with tree incident 
points or to $\tilde{J}$-holomorphic disks with incident points at the boundary and one 
incident point at $0$.
 \item The starting and ending points of the tree can also correspond to critical points 
of a function $\tilde{g}$ or $\tilde{h}$ on one of the ambient manifolds $\tilde{M}_T$ or 
$\tilde{M}_R$.
\end{enumerate}
Then we get the more general lemma:

\begin{lemma}~\label{lem:traj_awayfromboundary_moregeneral}
Let $\mathcal{P}$ be a moduli space modeled over planar trees as described above.
Suppose that the almost complex structure $\tilde{J}$ is such that the projection $\pi$ 
is $\tilde{J}$-holomorphic outside of $K\times M$.
Then, the elements of the moduli space $\mathcal{P}$ cannot reach the boundaries of $\tilde{M}_R$, $\tilde{M}_T$ or $V$.
In particular, if the virtual dimension of $\mathcal{P}$ is at most $1$, it is a smooth manifold of dimension equal to its virtual dimension and it is compact if it is $0$-dimensional.
Moreover, the compactification of a one dimensional such moduli space has the same description as the compactification
of the analogous moduli space as in the setting of a compact connected Lagrangian. 
\end{lemma}
\begin{proof}
Recall that Lemma~\ref{lem:curve_openmapping} also applies, if we have pseudo-holomorphic spheres instead of pseudo-holomorphic disks with boundary in a Lagrangian.
Therefore, the proof is similar to the proof of Lemma~\ref{lem:traj_awayfromboundary_general}.
\end{proof}

Lemma~\ref{lem:traj_awayfromboundary_moregeneral} ensures that the the quantum product on
$QH_*(\tilde{M}_T,\partial \tilde{M}_T)$ and $QH_*(\tilde{M}_R,\partial \tilde{M}_R)$ can 
be defined analogously as for the non-relative case.

\begin{lemma}
 For a generic choice of data, the operations~(\ref{eq:qprod_T}) and~(\ref{eq:qprod_R}) 
are well-defined and their linear extensions define chain maps. 
They induce products on the ambient quantum homology of $\tilde{M}_T$ and $\tilde{M}_R$ 
respectively, which are independent of the choices made in the constructions.
\end{lemma}

\begin{proof}
 By Lemma~\ref{lem:traj_awayfromboundary_moregeneral} the proof is equivalent to the proof of the existence of the quantum product for a closed manifold. 
This can 
be found in~\cite{McDuffSalamon} for example.
\end{proof}

As we will see below that the product on $QH_*(\tilde{M}_T,\partial\tilde{M}_T)$ actually turns out to be zero. 
The product on $QH_*(\tilde{M}_R,\partial \tilde{M}_R)$ can be described using the homology $QH_{*-2}(M)$.
These results are summarized in the next proposition.
\begin{lemma}\label{prop:Kunneth}
 The quantum homology $QH_*(\tilde{M}_R,\partial \tilde{M}_R)$ is isomorphic as a ring to 
the quantum homology of $M$ by a shift in degree.
More precisely
\begin{equation*}
 \begin{array}{cccc}
QH_*(\tilde{M}_R,\partial \tilde{M}_R)&\cong &QH_{*-2}(M).\\
 \end{array}
\end{equation*}
The quantum homology $QH_*(\tilde{M}_R,\partial \tilde{M}_R)$ is a unital ring.
The quantum product on $QH_*(\tilde{M}_T,\partial\tilde{M}_T)$ is trivial.
Additively it is isomorphic to the quantum homology $QH_{*-1}(M)$.
\end{lemma}

\begin{remark}\label{rmk:Kunneth}
 As we have already mentioned, the quantum homologies $QH_*(\tilde{M}_T,\partial 
\tilde{M}_T)$ and $QH_*(\tilde{M}_R,\partial\tilde{M}_R)$ are additively the same as the 
singular homologies tensored with $\Lambda$, i.e.  
$H_*(\tilde{M}_T,\partial \tilde{M}_T)\otimes \Lambda$ and 
$H_*(\tilde{M}_R,\partial\tilde{M}_R)\otimes \Lambda$ respectively.
Note that $H_i(\tilde{M}_T,\partial\tilde{M}_T)\cong H_{i-1}(M)\otimes_{\mathcal{R}} 
H_1(T,\partial T)\cong H_{i-1}(M)\otimes_{\mathcal{R}} {\mathcal{R}} \cong H_{i-1}(M)$ 
and $H_i(\tilde{M}_R,\partial\tilde{M}_R)\cong H_{i-2}(M)\otimes_{\mathcal{R}} H_2(R,\partial R) 
\cong H_{i-2}(M)\otimes_{\mathcal{R}} {\mathcal{R}}
\cong H_{i-2}(M)$ by the Künneth isomorphism.
Therefore, the isomorphism in Lemma~\ref{prop:Kunneth} can be described as follows.
\begin{equation*}
 \begin{array}{ccc}
\Phi_R:QH_{*-2}(M) & \xrightarrow{\cong} & QH_*(\tilde{M}_R,\partial \tilde{M}_R)\\
b & \mapsto & R\times b\\
\end{array},
\end{equation*}
where $b$ is a homology class in $QH_{*-2}(M)$ and $R$ represents the generator of 
$H_2(R,\partial R)$. 
In particular for every $R \times b$ and $R\times b'$ elements of 
$QH_*(\tilde{M}_R,\partial \tilde{M}_R)$ we have that
\begin{equation}\label{Phi_R}
 R\times(b*b')=\Phi_R(b * b')= \Phi_R(b) * \Phi_R(b') = (R \times b)* (R\times b').
\end{equation}
For $\tilde{M}_T$ we have
\begin{equation*}
\begin{array}{ccc}
\Phi_T:H_{*-1}(M) & \xrightarrow{\cong} & H_*(\tilde{M}_T,\partial \tilde{M}_T)\\
a & \mapsto
 & I\times a\\
\end{array},
\end{equation*}
where $a$ denotes some homology class in $QH_{*-1}(M)$ and $I$ denotes the generator of 
$H_*(T,\partial T)$, which is represented by the 
interval $[R_--\epsilon,R_++\epsilon]\times \{ 0 \} \subset T$.
\end{remark}
\begin{proof}[Proof of Lemma~\ref{prop:Kunneth}]
 By Remark~\ref{rmk:Kunneth} it is left to show that the first isomorphism $\Phi_R$ respects the quantum product.
Choose the almost complex structure on $\tilde{M}_R$ to be split, i.e. $\tilde{J}=i\oplus J$ for some almost complex structure $J$ on $M$.
In particular this means that the projection $\pi$ is everywhere $(\tilde{J},i)$-holomorphic.
Consider an element in the moduli space $\mathcal{M}_{prod}(x,y,z,\lambda;\tilde{h},\tilde{h}',\tilde{h}'' )$.
The projection of the $\tilde{J}$-holomorphic sphere, corresponding to the interior vertex of the tree, is constant. 
Hence this sphere is completely contained in one of the fibers of $\pi:T\times M \rightarrow T$, say $\pi^{-1}(t)$.
We may assume that $\tilde{h}=\tilde{h}''$ and that $\tilde{h}=f+ \tau$ and $\tilde{h}'=f'+ \tau'$. Here $f$ and $f'$ are Morse functions on $M$ and $\tau$ and $\tau'$ Morse functions on $T$, each with a unique maximum $\xi \in Crit(\tau)$ and $\xi'\in Crit(\tau')$ respectively.
In particular $\pi(x)=\pi(z)=\xi$.
Conclude that the part of the tree, corresponding to the negative gradient flow of $\tilde{h}$, also has to be completely contained in $\pi^{-1}(t)$.
In particular we see that $t=\xi$.
We may assume that $\xi$ is not a critical point of $\tau'$, then the gradient flow line of $-\nabla \tilde{h}'$ projects under $\pi$ to the unique negative gradient flow of $\tau'$ connecting $\xi'$ to $\xi$.

Let $x=(\xi, a)$, $y=(\xi', b)$ and $z=(\xi, c)$. 
The flow line of $-\nabla \tilde{h}'$ goes from $y$ to $u(e^{4\pi i/3})=:p$.
Suppose $p=(\xi, q)$.
By the choice of $h'$, the flow lines of $h'$ between $y$ and $p$ are in one to one correspondence with flow lines of $f'$ from $b$ to $q$.
Hence, we have a bijection between $\mathcal{M}_{prod}(x,y,z,\lambda)$ and $\mathcal{M}(a,b,c,\lambda)$.
This proves that the identity~(\ref{Phi_R}) holds.

For the proof of the second part of the proposition, consider two functions $\tilde{g}$ and $\tilde{g}'$ in general position.
Since the single critical point $\xi'$ of $\tau'$ has index one, we may choose the critical point $\xi$ of $\tau$ such that it is not contained in the unstable manifold of $\xi'$.
By this choice of $\tau$ and $\tau'$, the moduli space $\mathcal{M}_{prod}(x,y,z,\lambda;\tilde{g},\tilde{g'},\tilde{g})$ is empty. 
\end{proof}

It is easy to see that the inclusion $H_*(\tilde{M}_T,\partial\tilde{M}_T)\hookrightarrow H_*(\tilde{M}_R,\partial \tilde{M}_R)$ in singular homology is trivial.
The relation between the quantum homologies $QH_*(\tilde{M}_T,\partial \tilde{M}_T)$ and $QH_*(\tilde{M}_R,\partial \tilde{M}_R)$ is given by the following two corollaries.
\begin{corollary}
The inclusion map
\begin{equation*}
\begin{array}{ccc}
 QH_*(\tilde{M}_T,\partial \tilde{M}_T) & \rightarrow & QH_*(\tilde{M}_R,\partial \tilde{M}_R)\\
\end{array}
\end{equation*}
is trivial.
\end{corollary}
\begin{proof}
By Remark~\ref{rmk:Kunneth} we see that the inclusion is given by
\begin{equation*}
\begin{array}{ccc}
 I\times a & \mapsto & R \times a.\\
\end{array}
\end{equation*}
Clearly $I$ is zero inside $H_*(R,\partial R)$ which proves the corollary.
\end{proof}
\begin{lemma}
There exists a bilinear map
\begin{equation}\label{eq:incl_ambient}
 *:QH_*(\tilde{M}_T,\partial \tilde{M}_T)\otimes QH_*(\tilde{M}_R,\partial \tilde{M}_R) 
\rightarrow QH_*(\tilde{M}_T,\partial \tilde{M}_T),
\end{equation}
which endows $QH_*(\tilde{M}_T,\partial \tilde{M}_T)$ with the structure of a module over 
the ring $QH_*(\tilde{M}_R,\partial \tilde{M}_R)$.
Moreover this product can be described as follows.
For every $(I\times a)\in QH_*(\tilde{M}_T,\partial \tilde{M}_T)$ and $(R\times b)\in 
QH_*(\tilde{M}_R,\partial \tilde{M}_R)$ we have
\begin{equation}
 (I \times a) * (R\times b)=(I\times (a* b)).
\end{equation}
\end{lemma}
\begin{proof}
Consider the moduli spaces $\mathcal{M}_{prod}(x,y,z,\lambda;\tilde{g}, \tilde{h},\tilde{g'})$, which are defined in a similar way as the moduli spaces 
$\mathcal{M}_{prod}$ in the beginning of the section  with the only difference that the edges corresponding to the flow lines of $-\nabla \tilde{g}$, $-\nabla \tilde{h}$ and $-\nabla \tilde{g}'$ for Morse functions $\tilde{g}$ and $\tilde{g}'$ on $\tilde{M}_T$ and $\tilde{h}$ on $\tilde{M}_R$.
We may assume that $\tilde{g}=\tilde{g}'$.
By the same argument as in the proofs earlier in this paper we see that the the flow line corresponding to $-\nabla\tilde{g}$ and the $\tilde{J}$-holomorphic sphere in the core all map to the same point under the projection $\pi$.
In other words, they are all contained in one fiber.
Suppose $\tilde{g}=\tau_K+ f$, $\tilde{h}=\tau_R + f'$ and $x=(\xi, a)$, $y=(\xi', b)$ and $z=(\xi, c)$, where $\xi$ is the critical point of $\tau_K$ and $\xi'$ the critical point of $\tau_R$.
Since $\tilde{h}$ is of the form $\tau_R+ f'$, there exists a unique flow line of $\tau_R$ from $\pi(y)=\xi$ to the fiber containing the core and the flow lines of $-\nabla \tilde{g}$.
Hence there exists a bijection between the elements in $\mathcal{M}_{prod}(x,y,z,\lambda;\tilde{g},\tilde{h},\tilde{g'})$ and the elements in 
the moduli space $\mathcal{M}_{prod}(a,b,c,\lambda;f,f',f)$, which proves the lemma.
\end{proof}


We can define a module structure of the quantum homology $QH_*(V,S)$ 
over the unital ring $QH_*(\tilde{M}_R,\partial \tilde{M}_R)$.
\begin{definition}\label{def:Pmod}
Let $\tilde{f}:V\rightarrow \mathbb{R}$ be a Morse function on $V$ respecting the exit 
region $S$.
 Let $x,y$ be two critical points of $\tilde{f}:V\rightarrow \mathbb{R}$ and $a$ a critical point of $\tilde{h}:\tilde{M}_R\rightarrow \mathbb{R}$.
Let $\lambda \in H_2^D(\tilde{M}_R,V)$ be a class, possibly zero.
Consider the space of all sequences $(u_1,\dots u_l;k)$ of every possible length $l\geq 1$,where
\begin{enumerate}
 \item $1\leq k\leq l$.
 \item $u_i:(D,\partial D)\rightarrow (\tilde{M}_R,V)$ is a $\tilde{J}$-holomorphic disk for every $1\leq i\leq l$, which is assumed to be non-constant, except possibly if $i=k$.
 \item $u_1(-1)\in W^u_x(\tilde{f})$.
 \item For every $1\leq i \leq l-1$ there exists $0 <t_i<\infty$ such that $\phi^{\tilde{f}}_{t_i}(u_i(1))=u_{i+1}(-1)$.
 \item $u_l(1)\in W^s_y(\tilde{f})$.
 \item $u_k(0)\in W^u_a(\tilde{h})$.
 \item $[u_1]+ \dots +[u_k]=\lambda$.
\end{enumerate}
Two elements $(u_1,\dots u_l;k)$ and $(u'_1,\dots u'_{l'};k')$ in this space are viewed as equivalent if $l=l'$, $k=k'$ and for every $i\neq k$ there exists an automorphism $\sigma_i\in Aut(D)$ fixing $-1$ and $1$ such that $\sigma_i\circ u_i=u'_i$.
The space of all such elements $(u_1,\dots u_l;k)$ modulo the above equivalence relation is denoted 
$\mathcal{P}_{mod}(x,y,a,\lambda;\tilde{h},\tilde{\rho}_{\tilde{M}_R},\tilde{f}, \tilde{
\rho}_{V},\tilde{J})$.
\end{definition}
We define the virtual dimension $\delta(x,y,a,\lambda):=|x|+|a|-|y|+\mu(\lambda)-(2n+2)$.
On the chain level for $a\in Crit(\tilde{h})$ and $x\in Crit(\tilde{f})$ we define
\begin{equation}\label{eq:Pmod}
 a\ast x:=\smashoperator{\sum\limits_{ \begin{subarray}{c} y,\lambda\\ \delta_{mod}(x,y,a,\lambda)=0\end{subarray}}}
\sharp \mathcal{P}_{mod}(x,y,a,\lambda;\tilde{h},\tilde{\rho}_{\tilde{M}_R},\tilde{f},
\tilde{\rho}_{V},
\tilde{J}) yt^{\overline{\mu}(\lambda)},
\end{equation}
where the sum is taken over the critical points $y\in Crit(\tilde{f})$, such that the corresponding moduli spaces have dimension zero.

We have the following result
\begin{lemma}\label{prop:modulestructureV}
 For a generic choice of the data 
$\tilde{f},\tilde{\rho}_{V},\tilde{h},\tilde{\rho}_{\tilde{M}_R}$ the map in~(\ref{eq:Pmod}) is well-defined and a chain map.
It induces a map in homology
\begin{equation*}
 \ast: QH_*(\tilde{M}_R,\partial \tilde{M}_R)\otimes QH_*(V,S) \rightarrow QH_*(V,S),
\end{equation*}
which is independent of the choice of the data.
\end{lemma}

\begin{proof}[Proof of Lemma~\ref{prop:modulestructureV}]
Applying the Lemma~\ref{lem:traj_awayfromboundary_moregeneral} we see that the proof given in~\cite{QuantumStructures} adapts to our setting.
\end{proof}

There are four properties that are left to verify. They are stated in the following lemma.

\begin{lemma}
 For every $a,b \in QH(\tilde{M}_R,\partial \tilde{M}_R)$ and for every $x\in QH(V,S)$ the following properties are fulfilled.
\begin{itemize}
 \item $(a* b)\ast x=a* (b\ast x)$
 \item $e_{(\tilde{M}_R,\partial \tilde{M}_R)} \ast x=x$
\end{itemize}
 For any $a\in QH(\tilde{M}_R,\partial \tilde{M}_R)$ and any $x,y\in QH(V,\partial 
V)$ we have:
\begin{itemize}
 \item $a\ast (x*y)=(a\ast x)*y$
and
 \item $a\ast (x*y)=x*(a\ast y)$.
\end{itemize}
\end{lemma}

\begin{proof}
 For the proof see~\cite{QuantumStructures}. 
Lemma~\ref{lem:traj_awayfromboundary_moregeneral} shows once again that the proof directly generalizes to give a proof for the non-compact setting.
\end{proof}

\section{The Inclusion}\label{sec:inclusion}

The quantum homology $QH_*(\tilde{M}_R,\partial \tilde{M}_R)$ has the structure of a ring with unit.
The advantage of $QH_*(\tilde{M}_T,\partial \tilde{M}_T)$ over $QH_*(\tilde{M}_R,\partial \tilde{M}_R)$ lies in the fact that there exists a (possibly non-trivial) inclusion map of the Lagrangian quantum homology of the cobordism $QH_*(V,S)$ into $QH_*(\tilde{M}_T,\partial \tilde{M}_T)$. 
However, the inclusion of $QH_*(\tilde{M}_T,\partial \tilde{M}_T)$ into $QH_*(\tilde{M}_R,\partial \tilde{M}_R)$ is always trivial.

\begin{proposition}\label{prop:inclusion}
 There exists a $QH_*(\tilde{M}_R,\partial \tilde{M_R})$- linear inclusion map
\begin{equation*}
 i_{(V,S)}:QH_*(V,S) \to QH_*(\tilde{M}_T,\partial \tilde{M}_T),
\end{equation*}
which extends on the chain level the inclusion in singular homology.
\end{proposition}
\begin{proof}
We define the moduli spaces $\mathcal{P}_{inc}(x,a,\lambda; \tilde{f},\tilde{g},\tilde{J})$, which coincides with the definition in~\cite{QuantumStructures}, except that $\tilde{f}$ is a Morse function on $V$ adapted to the exit region $S$ and $\tilde{g}$ is a Morse function on $\tilde{M}_T$ as we defined it in section~\ref{sec:ambient_T_R}.
We put
\begin{equation*}
\begin{array}{ccccc}
 i:& C((V,S);\tilde{f},\tilde{\rho},\tilde{J})& 
\longrightarrow &
C((\tilde{M}_T,\partial \tilde{M}_T);\tilde{g},\tilde{\rho},\tilde{J})\\
		& x & \longmapsto & \smashoperator{\sum\limits_{\begin{subarray}{c}
						a,\lambda \\
 						\delta_{inc}(x,a,\lambda)=0\\
						\end{subarray}}}		
\sharp\mathcal{P}_{inc}(x,a,\lambda;\tilde{f},\tilde{g}) a t^{\overline{\mu}(\lambda)}\\
\end{array}.
\end{equation*}

If we restrict the sum in the above equation to moduli spaces with $\lambda=0$ then this is exactly the inclusion given in singular homology.
By Lemma~\ref{lem:traj_awayfromboundary_moregeneral}, the proof given in~\cite{QuantumStructures} still works for the non-compact case.

\end{proof}

\begin{remark}
 It is easy to see that $i(e_{(V,S)})=[(V,S)]\in QH(\tilde{M}_T,\partial \tilde{M}_T)$.
In particular if the inclusion is trivial then $[(V,S)]$ is trivial in $QH(\tilde{M}_T,\partial \tilde{M}_T)$.
\end{remark}

\begin{lemma}
 The inclusion map from Proposition~\ref{prop:inclusion} satisfies the property
\begin{equation}
 \langle h^*,i_L(x)\rangle=\epsilon(h*x),
\end{equation}
for every $x\in QH_*(L)$, $h\in QH_*(M)$. 
\end{lemma}

\begin{proof}
In~\cite{QuantumStructures} the analogous identity for the compact case is proven by showing that there exists a bijection between the moduli spaces defining the left hand side of the equation and the moduli spaces defining the right hand side.
By Lemma~\ref{lem:traj_awayfromboundary_moregeneral} this proof adjust to our setting.
\end{proof}

\section{Duality}\label{duality}

We start this section by introducing some notation, which can also be 
found in~\cite{QuantumStructures}.
Let $C^{\odot}:=hom_{\Lambda}(C((V,S);\tilde{f},\rho,J),\Lambda)$ such that the dual of 
$x$ has index $|x^*|:=-|x|$ and the differential is given by
\begin{equation}\label{eq:adjointdiff}
 \langle \partial^*g,x\rangle:= -(-1)^{|g|}\langle g,\partial x\rangle.
\end{equation}

We denote by $s^jC$ the $j$'th suspension of the complex $C$, i.e. $(s^jC)_k:=C_{k-j}$.
\begin{remark}\label{rmk:cohom} 
 Let $C^*$ denote the cochain complex, which is dual to $C$, namely $C^k:= 
Hom_{\mathcal{R}}(C_k,\mathcal{R})\otimes (\Lambda)^*$, where the grading is given by 
$|x^*|=|x|$ and the differential is the adjoint of the differential of $C$.
Notice that the chain complex $C^{\odot}$ is the same as $C^*$, only with opposite signs 
in the grading. 
Therefore we have an isomorphisms of the following form
\begin{equation*}
 H_k(s^{(n+1)}C^{\odot})\cong H_{k-(n+1)}(C^{\odot})\cong H^{(n+1)-k}(C^*).
\end{equation*}
In particular we define $QH^{(n+1)-k}(V,S)$ to be the $ k$'th cohomology of the cochain 
complex $C((V,S);\tilde{f},\rho,J,\Lambda)^*$.
$$QH^{(n+1)-k}(V,S):=H^k(C((V,S);\tilde{f},\rho,J,\Lambda)^*.)$$
\end{remark}

\begin{proposition}
Let $\tilde{f}'$ and $\tilde{f}$ be two Morse functions respecting the exit region $S$ 
and in general position.
 There exist a degree preserving chain morphism
\begin{equation*}
\eta:C((V,S);\tilde{f}',\rho,J)\rightarrow 
s^{(n+1)}(C(V,\partial V \setminus S);-\tilde{f},\rho,J))^{\odot},
\end{equation*}
which descends to an isomorphism in homology.
By remark~\ref{rmk:cohom} this induces an isomorphism
\begin{equation}
 \eta:QH_k(V,S)\rightarrow QH^{(n+1)-k}(V,\partial V \setminus S).
\end{equation}
The corresponding (degree $-(n+1)$) bilinear map
\begin{equation}
\begin{array}{ccc}
 \tilde{\eta}:QH(V,S) \otimes QH(V,\partial V \setminus S) & \rightarrow & \Lambda\\
x\otimes y & \mapsto & [\eta(x)(y)]
\end{array}
\end{equation}
coincides with $\epsilon_V (x *y)$.
Here, $\Lambda$ denotes the chain complex, with $\Lambda$ in degree zero, and zero else 
and with trivial differential.
In particular, $\tilde{\eta}(x\otimes y)=0$ if $|x|+|y|\not= n+1$.
\end{proposition}

\begin{proof}
We sketch the differences to the proof of the analogous statement for the compact case 
in~\cite{QuantumStructures}.
Let $\tilde{f}$ be a Morse function respecting the exit region $S$, then $-\tilde{f}$ is 
a Morse function respecting the exit region $\partial V \setminus S$. 
We have a basis preserving isomorphism $\iota$ between 
$C((V,S);\tilde{f},\rho,J,\Lambda)$ and $s^{(n+1)} C((V,\partial V \setminus 
S);-\tilde{f},\rho,J,\Lambda)^{\odot}$ defined as follows.
It takes a critical point $x$ of $\tilde{f}$ and sends it to the same critical point, now 
seen as a critical point of $-\tilde{f}$.
As a generator in $s^{(n+1)} C((V,\partial V \setminus 
S);-\tilde{f},\rho,J,\Lambda)^{\odot}$, the critical point $x$ has index 
$-(n+1-k)+(n+1)=k$.
With the sign convention in~(\ref{eq:adjointdiff}) this is a chain morphism.
We then compose $\iota$ with the comparison morphism 
$\phi:C((V, S);\tilde{f}',\rho,J,\Lambda)\rightarrow C((V, S);\tilde{f},\rho,J,\Lambda)$, 
which induces a canonical isomorphism in homology.
The composition $\eta:=\iota \circ 
\phi:C((V,S);\tilde{f}',\rho,J,\Lambda)\rightarrow s^{(n+1)} C((V,\partial V \setminus 
S);-\tilde{f},\rho,J,\Lambda)^{\odot}$ induces an isomorphism
$QH_k(V,S)\rightarrow QH^{(n+1)-k}(V,\partial V \setminus S)$. This proves the first part 
of the theorem.

In order to prove the second part of the theorem we need to define another moduli space. 
This moduli space is modeled on linear trees similar to the pearly trajectories, except 
that one edge corresponds to a marked point instead of a pseudo-holomorphic disk.
The linear tree connects a critical point $x$ of a Morse function $\tilde{f}'$ respecting 
the exit region $S$ to a critical point $y$ of a Morse function $\tilde{f}$ also 
respecting the exit region $S$.
Then the edges between $x$ and the marked point are labeled by the negative gradient flow 
lines of $\tilde{f}'$ and the edges between the marked point and $y$ are labeled by the 
negative gradient flow lines of $\tilde{f}$.
We call these moduli spaces $\mathcal{P}^{!}(x,y;\tilde{f}',\tilde{f})$.
As in~\cite{QuantumStructures} one can see that the moduli spaces 
$\mathcal{P}^{!}(x,y;\tilde{f}',\tilde{f})$ induce a chain map $\phi'$ which is chain 
homotopic to the comparison map $\phi$.
Therefore, counting the zero dimensional moduli 
spaces $\mathcal{P}^{!}(x,y;\tilde{f}',\tilde{f})$ gives $<\phi'(x),y>=<\phi(x),y>=i\circ 
\phi(x)(y)$.

Let $\tilde{f''}$ be a Morse function respecting the exit region $\emptyset$. 
In particular we may assume that $\tilde{f''}$ has a unique minimum $m$ inside $V$.
Consider the moduli spaces $\mathcal{P}_{prod}(x,y,m;\tilde{f}',-\tilde{f},\tilde{f}'')$.
The zero-dimensional moduli spaces of this sort compute exactly $\epsilon_V(x*y)$.

As in~\cite{QuantumStructures} we argue that in dimension zero, the moduli spaces 
$\mathcal{P}_{prod}(x,y,m;\tilde{f}',-\tilde{f},\tilde{f}'')$ and 
$\mathcal{P}^{!}(x,y;\tilde{f}',\tilde{f})$ are in bijection. 
Indeed, if $dim \mathcal{P}_{prod}(x,y,m)=0$, the central disk with valence three is 
constant and there is a unique flow line of $-\nabla \tilde{f}''$ from this point to $m$.
\end{proof}

\section{Proof of Theorem~\ref{thm:longexact}}\label{chap:longexact}

We start by proving the existence of the long exact sequence in homology.
For this we define a special Morse functions on $(V,S)$.
Recall that we assume the following.
On the cylindrical ends of $V^{\epsilon}$ the almost complex structure $\tilde{J}$ splits into $i\oplus J$ for some $J$ on $M$ and $i$ the standard complex structure on $\mathbb{R}^2$.
Likewise on the cylindrical ends of $V^{\epsilon}$, the metric $\rho$ is of the form $\rho^{\pm}\oplus \rho_M$, for some metric $\rho^{\pm}$ on $\coprod L_i^-$ and $\coprod L_j^+$.
\begin{definition}\label{def:fadaptedS}
Let $S$ be the union of some of the ends of $V$, i.e. 
$$S= (\coprod_{i \in I_-} \{(R_-,a_i^-)\}\times L_i^-) \cup (\coprod_{j\in J_+} \{(R_+,a_j^+)\}\times L_j^+),$$
where $I_-\subset \{1,\dots k_-\}$ and $J_+\subset \{1,\dots k_+\}$.
Denote by 
$$S^{\epsilon}:=(\coprod_{i \in I_-} \{(R_- - \epsilon,a_i^-)\}\times L_i^-) \cup  (\coprod_{j\in J_+} \{(R_+ + \epsilon, a_j^+)\}\times L_j^+)$$ 
the corresponding union of connected components of $\partial V^{\epsilon}$.
Let $\tilde{f}: V^{\epsilon} \rightarrow \mathbb{R}$ be a Morse function on  $V^{\epsilon}$, such that $-\nabla \tilde{f}$ is transverse to the boundary $\partial V^{\epsilon}$ and points outside along $S^{\epsilon}$ and inside along $\partial V^{\epsilon} \setminus S^{\epsilon}$.
Moreover we want $\tilde{f}$ to fulfill the following properties.
\begin{equation*}
        \begin{array}{lll}
        \tilde{f}(t,a_j^+,p)=f_j^+(p)+\sigma_j^+(t) & 
\sigma_j^+:[R_+,R_+ +\epsilon]\rightarrow \mathbb{R}, & p\in M, j=1,\dots, k_+, \\
	\tilde{f}(t,a_i^-,p)=f_i^-(p)+\sigma_i^-(t) & \sigma_i^-:[R_--\epsilon,R_- ]\rightarrow
\mathbb{R}, & p\in M, i=1,\dots, k_-, \\
       \end{array},
\end{equation*}
where $f_j^+:L_j^+\rightarrow \mathbb{R}$ and $ f_i^-:L_i^- \rightarrow \mathbb{R}$ are Morse functions on $L_j^+$ and $L_i^-$ respectively.
The functions $\sigma_j^+$ and $\sigma_i^-$ are also Morse, each with a unique critical point and satisfying
\begin{enumerate}
  \item [(i)] $\sigma_j^+(t)$ is a non-constant linear function for $t \in [R_+ + \frac{3\epsilon}{4},R_+ + \epsilon]$ and $\sigma_j^+(t)$ is decreasing in this interval if $j\in J_+$ and increasing if $j\in  \{1,\dots,k_+\}\setminus J_+$.
Furthermore $\sigma_j^+(t)$ has a unique critical point at $R_+ + \frac{\epsilon}{2}$ of index $1$ if $i\in J_+$ and of index $0$ if $j\in \{1,\dots,k_+\}\setminus J_+$
  \item [(ii)] $\sigma_i^-(t)$ is a non-constant linear function for $t \in [R_- -\epsilon, R_- - \frac{3\epsilon}{4}]$ and $\sigma_i^-(t)$ is increasing in this interval if $i\in I_-$ and decreasing if $i\in  \{1,\dots,k_-\}\setminus I_-$.
Furthermore, $\sigma_i^-(t)$ has a unique critical point at $R_- - \frac{\epsilon}{2}$ of index $1$ if $i\in I_-$ and of index $0$ if $i\in \{1,\dots,k_-\}\setminus I_-$
\end{enumerate}
We call such a function on $V^{\epsilon}$ a Morse function adapted to the exit region $S^{\epsilon}$.
More generally, we denote $S^{\eta}:=(\coprod_{i \in I_-} \{(R_- - \eta,a_i^-)\}\times L_i^-) \cup (\coprod_{j\in J_+} \{(R_+ + \eta,a_j^+)\}\times L_j^+)$, for $0\leq\eta\leq\epsilon$.
Let $N$ be the neighborhood of $S^{\epsilon}$ given by
\begin{equation*}
 N:=\coprod_{i \in I_-} [R_- -\epsilon,  R_- -3\epsilon/8]\times \{a^-_i\} \times L_i^-\coprod_{j\in J_+} [R_++3\epsilon/8,  R_+ +\epsilon]\times \{a^+_j\} \times L_j^+.
\end{equation*}
Then the set $U$ is defined by $U:=V^{\epsilon}\setminus N$.
\end{definition}
For once let us for forget about the orientations of the moduli spaces and assume that $\mathcal{R}=\mathbb{Z}_2$.
With the above definition of a Morse function on the cobordism we know from~\cite{QuantumStructures} that there exists a short exact sequence of chain complexes

\begin{equation}\label{eq:shortexactchain}
\xymatrix{
0 \ar[r] & C_k(U;\tilde{f}|_U,\tilde{J};\mathbb{Z}/2) \ar[r]^j \ar[d]^{d_U} & C_k((V,S);\tilde{f},\tilde{J};\mathbb{Z}/2) \ar[r]^{\delta} \ar[d]^{d_{(V,S)}} & C_{k-1}(S^{\epsilon/2};\tilde{f}|_{S^{\epsilon/2}},J;\mathbb{Z}/2) \ar[d]^{d_{S^{\epsilon/2}}} \ar[r] & 0\\
0 \ar[r] & C_{k-1}(U;\tilde{f}|_U,\tilde{J};\mathbb{Z}/2) \ar[r]^{j} & C_{k-1}((V,S);\tilde{f},\tilde{J};\mathbb{Z}/2) \ar[r]^{\delta} & C_{k-2}(S^{\epsilon/2};\tilde{f}|_{S^{\epsilon/2}},J;\mathbb{Z}/2) \ar[r] & 0 \\
}
\end{equation}
This short exact sequence induces a long exact sequence
\begin{equation*}
\xymatrix{
 \dots \ar[r]^{\delta_*} & QH_*(S) \ar[r]^{i_*} & QH(V)\ar[r]^{j_*}
& QH_*(V,S)\ar[r]^{\delta_*} & QH_{*-1}(S)\ar[r]^{i_*} & \dots
}
\end{equation*}
More generally, if $A$ and $B$ are two subsets of the ends of $V$, such that $A\cap B=\emptyset$, then there exists a long exact sequence
\begin{equation*}
\xymatrix{
 \dots \ar[r]^{\delta_*} & QH_*(A) \ar[r]^{i_*} & QH(V,B)\ar[r]^{j_*}
& QH_*(V,A)\ar[r]^{\delta_*} & QH_{*-1}(A)\ar[r]^{i_*} &\dots
}
\end{equation*}
\begin{remark}
 Notice that the connecting homomorphism in homology coming from the short exact sequence of chain complexes~(\ref{eq:shortexactchain}) gives us the inclusion map $i_*:QH_*(S)  \rightarrow QH(V)$.
In singularhomology the connecting homomorphism is the one between the homologies $H_*(V,S)$ and $H_{*-1}(S)$ and it is often denoted by $\delta$.
However, in the case of quantum homology the map $\delta_*$ is not the connecting homomorphism.
\end{remark}
When working with a general ring $\mathcal{R}$, which has $Char(\mathcal{R}) \neq 2$, it is harder to see that the maps $j$ and $\delta$ above are chain maps. 
Clearly $j$ is a chain map, since $C_k(U;\tilde{f}|_U,\tilde{J})$ is a subcomplex of $C_k((V,S);\tilde{f},\tilde{J})$.
For the map $\delta$ this is less obvious.
On the one side, the chain complex $C(V,S)$ is defined using moduli spaces that are fiber products of stable and unstable manifolds of some Morse function $\tilde{f}$ on $V^{\epsilon}$ and moduli spaces of $\tilde{J}$-holomorphic disks in $(\tilde{M},V)$. 
On the other side, the chain complex $C(S^{\epsilon/2};\tilde{f}|_{S^{\epsilon/2}};J)$ is defined using moduli spaces that are fiber products of stable and unstable manifolds of a Morse function $f$ on $S$ and moduli spaces of $J$-holomorphic disks in $(M,S)$.
To calculate the map $\delta_*$ we need to identify some of these moduli spaces.
However, it is not clear that under these identifications the orientations of the moduli spaces match.
This is the content of the next section.

\subsection{Orientations of the Moduli Spaces on the Boundary of 
\texorpdfstring{$V$}{V}}\label{sec:orientation}
Recall that the orientations of the moduli spaces of pseudo-holomorphic disks are defined using a fixed spin structure on the Lagrangian submanifold.
We use the following definition of a spin structure, which is due to Milnor~\cite{Milnor} and equivalent to the usual definition (for example the definition in~\cite{SpinGeometry}).
\begin{definition}
A spin structure of an oriented vector bundle $E$ over a manifold $X$ is a homotopy class of a trivialization of $E$ over the 1-skeleton of $X$ which can be extended to the 2-skeleton of $X$.
\end{definition}
With this definition of a spin structure on a general vector bundle we can define the spin structure of a manifold as follows.
\begin{definition}
 A spin manifold is a oriented Riemannian manifold with a spin structure on its tangent bundle.
\end{definition}
Given a spin structure on the Lagrangian cobordism $V$ there is a canonical way to define a spin structure on its boundary components.
For a description the reader is referred to~\cite{SpinGeometry}.
Let $\mathcal{M}(\lambda;\tilde{J};(\tilde{M},V))$ be the moduli spaces of $\tilde{J}$-holomorphic disks in $(\tilde{M},V)$ and similarly let $\mathcal{M}(\lambda;\tilde{J};(M,S))$ be the moduli space of $J$-holomorphic disks in $(M,S)$.
For a precise definition we refer the reader to~\cite{McDuffSalamon}.
We endow $L$ with the spin structure induced by the spin structure of $V$. 
This orients the space $\mathcal{M}(\lambda;\tilde{J};(M,S))$.
The following lemma guarantees that the orientations of the moduli spaces $\mathcal{M}(\lambda;\tilde{J};(\tilde{M},V))$ and $\mathcal{M}(\lambda;\tilde{J};(M,S))$ match.
\begin{lemma}\label{lem:orientationmatch}
 The orientation of $\mathcal{M}(\lambda;\tilde{J};(\tilde{M},V))$ restricted to the set of $\tilde{J}$-holomorphic curves contained in $(M,L)$ is the same as the orientation of $\mathcal{M}(\lambda;\tilde{J};(M,S))$.
\end{lemma}
In order to prove this lemma we need to briefly recall the construction of the orientation of $\mathcal{M}(\lambda;\tilde{J};(\tilde{M},V))$.
Consider the following proposition, which is taken from~\cite{LagrangianIntersectionFloerTheory}.
\begin{proposition}\label{prop:indexbundle}
 Let $E$ be a complex vector bundle over a disk $D^2$. Let $F$ be a totally real subbundle of $E|_{\partial D^2}$ over $\partial D^2$. 
We denote by $\overline{\partial}_{(E,F)}$ the Dolbeault operator on $D^2$ with coefficients in $(E,F)$,
\begin{equation}
 \overline{\partial}_{(E,F)}:W^{(1,p)}(D^2,\partial D^2;E,F)\rightarrow L^{p}(D^2;E).
\end{equation}
Assume $F$ is trivial and take a trivialization of $F$ over $\partial D^2$.
Then the trivialization induces a canonical orientation of the index bundle $Ker \overline{\partial}_{(E,F)}-Coker \overline{\partial}_{(E,F)}$.
\end{proposition}
If a Lagrangian submanifold $L\subset M$ is oriented then $(u|_{\partial D^2})^*TL$ is trivial.
Indeed, $TL$ is trivial over the $1$-skeleton of $V$ and by a cellular approximation argument we may assume that $TL$ is trivial over the image of $u|_{\partial D^2}$. 
Then the first Stiefel-Whitney class $w_1(TL)$ vanishes, which implies that $w_1(u^*TL)=u^*(w_1(TL))=w_1(TL)=0$ and hence $(u|_{\partial D^2})^*TL$ is trivial.
Thus we can apply proposition~\ref{prop:indexbundle} to the case $(E,F)=(u^*TM,(u|_{\partial D^2})^*TL)$.
This gives us a pointwise orientation of the index bundle of the Dolbeault operator $\overline{\partial}_{(E,F)}$.
However, this trivialization is not unique, but the choice of a spin structure on $V$ induces a trivialization that is unique up to homotopy.
Since the orientation of the index bundle of $\overline{\partial}_{(E,F)}$ depends only on the trivialization of $(u|_{\partial D^2})^*TL$, it therefore depends only on the spin structure on $L$.
In~\cite{LagrangianIntersectionFloerTheory} it is shown that with a choice of a spin 
structure on $L$ the pointwise
orientation from the proposition~\ref{prop:indexbundle}
can be extended in a unique way to give a consistent orientation of the index bundle.
The orientation of the index bundle then induces an orientation of the determinant bundle 
of the linearized operator
$D\overline{\partial}_u$ of the $\tilde{J}$-holomorphic curve equation.
Recall that the determinant bundle of a Fredholm operator is defined as
\begin{equation*}
 det(D \overline{\partial}_u):=det(Coker D\overline{\partial}_u)^*\otimes det( Ker D 
\overline{\partial}_u).
\end{equation*}

\begin{proof}[Proof of Lemma~\ref{lem:orientationmatch}]
By the construction of the orientation of the moduli spaces in~\cite{LagrangianIntersectionFloerTheory}, it suffices to show that the orientations of the determinant line bundles of the linearized operators match.
Let $\tilde{u}\in \mathcal{M}(\lambda;\tilde{J};(\tilde{M},V))$ be a $\tilde{J}$-holomorphic disk that is mapped to the complement of
$K\times M$. In particular we may assume that $\tilde{J}=i\oplus J$ and $\pi(\tilde{u})$ is a constant.
Then the linearization of the operator $D \overline{\partial}_{\tilde{u}}$ splits into $D
\overline{\partial}_{u_{\mathbb{C}}}\oplus D\overline{\partial}_{u}$, where $u\in \mathcal{M}(\lambda;J;(M,S))$ and $u_{\mathbb{C}}$ is a constant curve in $\mathbb{C}$.
We compute the determinant line bundle:
\begin{equation*}
\begin{array}{ccc}
 det \left( D \overline{\partial}_{\tilde{u}}\right)& = &det\left(Coker D\overline{\partial}_{\tilde{u}}\right)^*\otimes det\left( KerD\overline{\partial}_{\tilde{u}}\right)\\
 &=& det\left(Coker D\overline{\partial}_{u_{\mathbb{C}}}\oplus Coker D\overline{\partial}_u\right)^*\otimes det\left(Ker D\overline{\partial}_{u_{\mathbb{C}}}\oplus Ker D\overline{\partial}_u\right)\\
 &=& det\left( Coker D\overline{\partial}_u\right)^*\otimes det\left(Coker D\overline{\partial}_{u_{\mathbb{C}}}\right)^* \otimes
 det\left(Ker D\overline{\partial}_{u_{\mathbb{C}}}\right)\otimes det\left( Ker D\overline{\partial}_u\right)\\
 &=& det\left(Coker D\overline{\partial}_{u_{\mathbb{C}}}\right)^* \otimes
 det\left(Ker D\overline{\partial}_{u_{\mathbb{C}}}\right)\otimes det\left( Coker D\overline{\partial}_u\right)^*\otimes det\left( Ker D\overline{\partial}_u\right),\\
\end{array}
\end{equation*}
where the last step follows, since $D\overline{\partial}_u$ is surjective and therefore the cokernel is trivial.
Hence we have the identity
\begin{equation*}
 det \left(D \overline{\partial}_{\tilde{u}}\right)=det\left(D\overline{\partial}_{u_{\mathbb{C}}}\right)\otimes det\left(D \overline{\partial}_u\right).
\end{equation*}
Restricting the trivialization of $(\tilde{u}|_{\partial D^2})^*TV$ to a trivialization of $(u|_{\partial D^2})^*TL$, implies that an orientation of the determinant bundle of $D \overline{\partial}_{\tilde{u}}$ induces an orientation of $D \overline{\partial}_u$.
But this trivialization is exactly the trivialization induced by the spin structure on $L$. 
Therefore the two orientations are equivalent.
\end{proof}

Let $e_1:U\rightarrow X$ and $e_2: V\rightarrow X$ be smooth maps between oriented manifolds, that are transverse.
We denote by $U\sideset{{}_{e_1}}{{}_{e_2}}{\mathop{\times}}V:=\{(u,v)\subset U\times V|e_1(u)=e_2(v)\}$ the fiber product of $U$ and $V$ with the orientation induced by the orientations of $U$, $V$ and $X$.

\begin{remark}
Let $e_1:U\rightarrow X$ and $e_2: V\rightarrow X$ be smooth maps and $X$ a submanifold in $Y$. 
Denote by $i:X\rightarrow Y$ the inclusion and write $e_1':U\rightarrow Y$ and $e_2':V \rightarrow Y$ for the compositions $i\circ e_1$ and $i\circ e_2$ respectively.
Then:
\begin{equation}\label{eq:orientationFiberProduct}
 U\sideset{{}_{e_1}}{{}_{e_2}}{\mathop{\times}}V=
U\sideset{{}_{e_1'}}{{}_{e_2'}}{\mathop{\times}}V.
\end{equation}
\end{remark}

\begin{corollary}\label{cor:orientationsmoduli}
There are two ways to orient moduli spaces modeled over planar trees in $\partial V$.
One is induced by the spin structure on $V$ and the other one is induced by the spin structure on $L$.
These two orientations are the same.
\end{corollary}

\begin{proof}
 With the conventions $W_f^u(x)=W_{\tilde{f}}^u(x)\cap L$, identity~(\ref{eq:orientationFiberProduct}) and Lemma~\ref{lem:orientationmatch} the corollary follows.
\end{proof}

\subsection{The Long Exact Sequence}

\begin{proposition}\label{prop:longexact}
For a general ring $\mathcal{R}$, there exists a short exact sequence of chain complexes
\begin{equation}\label{eq:shortexactchain_Z}
\xymatrix{
0 \ar[r] & C_k(U;\tilde{f}|_U,\tilde{J}) \ar[r]^j \ar[d]^{d_U} & C_k((V,S);\tilde{f},\tilde{J}) \ar[r]^{\delta} \ar[d]^{d_{(V,S)}} & C_{k-1}(S^{\epsilon/2};\tilde{f}|_{S^{\epsilon/2}},J) \ar[d]^{d_{S^{\epsilon/2}}} \ar[r] & 0\\
0 \ar[r] & C_{k-1}(U;\tilde{f}|_U,\tilde{J}) \ar[r]^{j} & C_{k-1}((V,S);\tilde{f},\tilde{J}) \ar[r]^{\delta} & C_{k-2}(S^{\epsilon/2};\tilde{f}|_{S^{\epsilon/2}},J) \ar[r] & 0 \\
}
\end{equation}
which induces the long exact sequence in homology
\begin{equation*}
\xymatrix{
 \dots \ar[r]^{\delta_*} & QH_*(S) \ar[r]^{i_*} & QH(V) \ar[r]^{j_*}
& QH_*(V,S) \ar[r]^{\delta_*} & QH_{*-1}(S)\ar[r]^{i_*} &\dots
}.
\end{equation*}
\end{proposition}

\begin{proof}[Proof of Proposition~\ref{prop:longexact}]
Let $\tilde{f}$ be a Morse function adapted to the exit region $S^{\epsilon}$ of $V^{\epsilon}$. 
We may assume that $K\times M\subset U$.
The map $j$ in~(\ref{eq:shortexactchain_Z}) is given by the inclusion and $\delta$ is the restriction map.
By the discussion above, the map $\delta$ is also a chain map.
Hence all the maps in~(\ref{eq:shortexactchain_Z}) are chain maps.
This induces the long exact sequence in homology.
\end{proof}

In this section we summarize some properties of the maps in the long exact sequence in Proposition~\ref{prop:longexact}, which is the content of Theorem~\ref{thm:longexact}.
First of all we define a Morse function, which is a small modification of Definition~\ref{def:fadaptedS}.
The functions $\sigma_j^+$ and $\sigma_i^-$ at the cylindrical ends of the cobordism corresponding to the set $S^{\epsilon}$ are defined with two instead of one critical point.

\begin{definition}\label{def:specialMF}
Let $\tilde{f}$ be a Morse function on $V^{\epsilon}$ as in Definition~\ref{def:fadaptedS}, with the only difference that $\sigma_j^+$ and $\sigma_i^-$ satisfy
\begin{enumerate}
  \item [(i)] $\sigma_j^+(t)$ is a non-constant linear function for $t \in [R_+ + \frac{3\epsilon}{4},R_+ + \epsilon]$.
$\sigma_j^+(t)$ is decreasing in this interval if $j\in J_+$ and increasing if $j\in \{1,\dots,k_+\}\setminus J_+$.
Furthermore if $j\in J_+$ then $\sigma_j^+(t)$ has a critical point $t_{j,1}^+:=R_+ + \frac{\epsilon}{2}$ of index $1$, and a critical point $t_{j,0}^+:=R_+ + \frac{\epsilon}{4}$ of index $0$.
If $j\in\{1,\dots,k_+\}\setminus J_+$ then $\sigma_j^+$ has a exactly one critical point $R_+ + \frac{\epsilon}{2}$, which is of index zero.
  \item [(ii)] $\sigma_i^-(t)$ is a non-constant linear function for $t \in [R_- -\epsilon, R_- -\frac{3\epsilon}{4}]$.
$\sigma_i^-(t)$ is increasing in this interval if $i\in I_-$ and decreasing if $i\in \{1\dots,k_-\}\setminus I_-$.
Furthermore if $i\in I_-$ then $\sigma_i^-(t)$ has a critical point $t_{i,1}^-:=R_- - \frac{\epsilon}{2}$ of index $1$ and a critical point $t_{i,0}^-:=R_- -\frac{\epsilon}{4}$ of index $0$. 
If $i\in \{1\dots,k_-\}\setminus I_-$ then $\sigma_i$ has a exactly one critical point $R_- - \frac{\epsilon}{2}$ of index $0$.
\end{enumerate}
Now we define the sets $N$, $U$ and $K$ similarly as in the proof of Proposition~\ref{prop:longexact}.
\end{definition}

Throughout the rest of this section we will use the following notation.
On a cylindrical end $[R_- -\epsilon, R_-]\times \{a_i^-\}\times L_i^-$ the critical points are of the form $(t_{i,1}^+,a_i^-,p)$ or $(t_{i,0}^-,a_i^-,p)$, where $p$ is a critical point of $f_i^-$.
Similarly on a cylindrical end $[R_+, R_+ +\epsilon]\times \{a_j^+\}\times L_j^+$ the critical points are of the form $(t_{j,1}^+,a_j^+,p)$ or $(t_{j,0}^+,a_j^+,p)$, where $p$ is a critical point of $f_j^+$.
If we do not want to specify the connected component of $S$, we will sometimes write $(t_{\iota,0},a_\iota,p)$ and $(t_{\iota,1},a_{\iota},p)$, as well as $f_{\iota}$ and $L_{\iota}$ for $\iota\in I_- \cup J_+$.
The functions $\sigma_i^-$ and $\sigma_j^+$ are illustrated in figure~\ref{sigma_2}.
Notice that $S^{\epsilon/4}\subset U$.
\begin{figure}
 \centering
\includegraphics{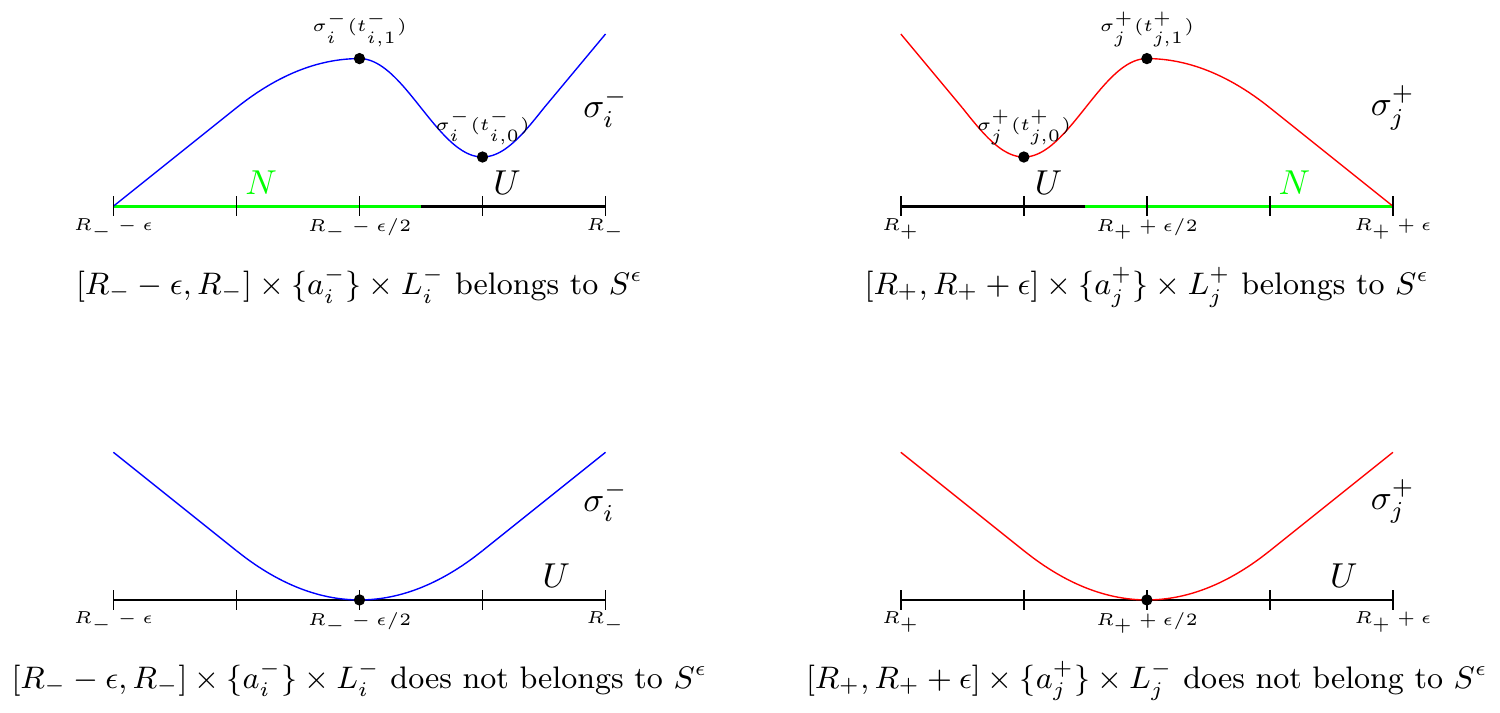}
\caption{The functions $\sigma_i^-$ and $\sigma_j^+$.}\label{sigma_2}
 \end{figure}
Consider the two sets $Crit(\tilde{f})\cap S^{\epsilon/2}$ and $Crit(\tilde{f})\cap S^{\epsilon/4}$, where $\tilde{f}$ is a function as it is given in Definition~\ref{def:specialMF}.
Then there is a natural identification between these two sets with a shift in degree by one and given by $(t_{\iota,1},a_{\iota},p) \mapsto (t_{{\iota},0},a_{\iota},p)$.
Note that the index of $(t_{{\iota},1},a_{\iota},p)$ is $|p|+1$ and the index of $(t_{{\iota},0},a_{\iota},p)$ is $|p|$.

\begin{lemma}\label{lem:sumofunits}
Suppose that $S=\partial V$. Then $\delta_*(e_{(V,\partial V)})= \oplus_{i} e_{L_i^-} \oplus_{j} e_{L_j^+}$.
\end{lemma}
\begin{proof}
We choose a Morse function $\tilde{f}$ as in Definition~\ref{def:specialMF} with $S=\partial V$.
Then $-\nabla \tilde{f}$ is still transverse to $V^{\epsilon}$ and points outside along $\partial V^{\epsilon}$.
Suppose $\tilde{f}$ has a unique maximum $m$ inside the set $V^{\epsilon}|_K$, which is 
possible by Proposition~\ref{lem:singlemax_rel}.
In addition assume that each $f_i^{-}$ and $f_j^{+}$ are Morse functions on $L_i^{-}$ and $L_j^{-}$, both with unique maxima $m_i^{-}$ and $m_j^{+}$ respectively.
Then 
$$\tilde{m}_{i,1}^{-}:=(t_{i,1}^-,a_i^-,m_i^{-}) \text{ and } \tilde{m}_{j,1}^{+}:=(t_{j,1}^+,a_j^+,m_j^{+})$$
 are also maxima of $\tilde{f}$.
We want to show that 
$$\sum\limits_i \tilde{m}_{i,1}^{-} +\sum\limits_j \tilde{m}_{j,1}^{+}-m$$
 is a cycle and hence 
$$[\sum\limits_i \tilde{m}_{i,1}^{-} + \sum\limits_j \tilde{m}_{j,1}^{+}]=[m]\in QH_*(V,\partial V).$$
Fix a critical point $\tilde{m}_{\iota,1}=(t_{\iota,1},a_{\iota},m_{\iota})\in\{ \tilde{m}_{i,1}^{-} , \tilde{m}_{j,1}^{+}\}_{i,j}$ and compute $d(\tilde{m}_{\iota,1})$.
By the choice of the $\sigma_j^+$ and $\sigma_i^-$ we know that any pearly trajectory from a critical point in $Crit(\tilde{f})\cap S^{\epsilon/2}$ cannot get into the set $V|_{[R_-, R_+]\times \mathbb{R}}$.
Hence it has to end at a point in $Crit(\tilde{f})\cap ({S^{\epsilon/2}}\cup S^{\epsilon/4})$.
Moreover, $m_{\iota}$ is a cycle in the complex $C(S^{\epsilon/2};\tilde{f}|_{S^{\epsilon/2}},J)$, as we saw in the proof of Lemma~\ref{lem:unit_rel}.
Therefore it suffices to consider pearly trajectories in $0$-dimensional moduli spaces from $\tilde{m}_{\iota,1}\in Crit(\tilde{f})\cap S^{\epsilon/2}$ to a critical point $x:=(t_{\iota,0},a_{\iota},p)\in Crit(\tilde{f})\cap S^{\epsilon/4}$.
Then we have 
$$\delta_{prl}(\tilde{m}_{\iota},x,\lambda;\tilde{f})=|\tilde{m}_{\iota}|-|x|+\mu(\lambda)-1=0.$$
Since outside of the set $K\times M$ the Morse function $\tilde{f}$ is the sum of $f_k$ and $\sigma_k$, we can project to the last factor to get a pearly trajectory of $f_{\iota}$ in $L_{\iota}$, going from $m_{\iota}$ to $p$.
Assume $p\neq m_{\iota}$.
Computing the dimension of moduli space of the projected pearly trajectory, we get 
$$\delta_{prl}(m_{\iota},p,\lambda;f_{\iota})=|m_{\iota}|-|p|+\mu(\lambda)-1=(|\tilde{m}_{\iota}|-1)-|x|+\mu(\lambda)-1=-1,$$
 which is a contradiction.
Thus, $d(\tilde{m}_{{\iota},1})=\tilde{m}_{{\iota},0}$ and therefore 
$$d(\sum\limits_i \tilde{m}_{i,1}^{-} + \sum\limits_j \tilde{m}_{j,1}^{+})= \sum\limits_i \tilde{m}_{i,0}^{-} +\sum\limits_j \tilde{m}_{j,0}^{+}.$$

It is left to show that $d(m)=\sum\limits_i {\tilde{m}_{i,0}^{-}} + \sum\limits_j {\tilde{m}_{j,0}^{+}}$.
A computation shows that for $\mu(\lambda)\neq0$ 
$$\delta_{prl}(m,z,\lambda;\tilde{f})=(n+1)-|z|+\mu(\lambda)-1\geq (n+1)-n+2>0.$$
Thus, $d(m)$ is equal to the Morse part of the differential, and hence $|z|=n$.
At each critical point of index $n$ two trajectories of the negative gradient of $\tilde{f}$ end.
If $z$ is one of the critical points of index $n$ on $S^{\epsilon/4}$, then one negative gradient is coming from $\tilde{m}_{\iota}$ for some $\tilde{m}_{\iota}\in \{\tilde{m}_{i,1}^{-},\tilde{m}_{j,1}^{+}\}_{i,j}$, and the other one therefore has to come from $m$, because it is the only maximum within $V|_{[R_-, R_+]\times \mathbb{R}}$ and $-\nabla \tilde{f}$ is transverse to $\partial V|_{[R_-, R_+]\times \mathbb{R}}$ and pointing outside.
Clearly the moduli spaces $\mathcal{P}_{prl}(m,\tilde{m}_{{\iota},0};0)$ and $\mathcal{P}_{prl}(\tilde{m}_{\iota},\tilde{m}_{{\iota},0};0)$ must have opposite orientations.
If $z$ is contained in $V|_{[R_-, R_+]\times \mathbb{R}}$, then there are two flow lines from $m$ to $z$, counted with opposite signs.
This proves that 
$$d(m)=\sum\limits_i \tilde{m}_{i,0}^{-} + \sum\limits_j \tilde{m}_{j,0}^{+}= d(\sum\limits_i \tilde{m}_{i,1}^{-} +\sum\limits_j \tilde{m}_{j,1}^{+})$$
 and thus 
$$e_{(V,\partial V)}=[m]=[\sum\limits_i \tilde{m}_{i,1}^{-} + \sum\limits_j \tilde{m}_{j,1}^{+}].$$
Now it is easy to see that 
$$\delta_*(e_{(V,\partial V)})=\delta_*([\sum\limits_i \tilde{m}_{i,0}^{-} + \sum\limits_j \tilde{m}_{j,0}^{+}])=\sum\limits_i [m_i^{-}] + \sum\limits_j [m_j^{+}]= \oplus_{i} e_{L_i^-}\oplus_{j} e_{L_j^+}.$$
This proves the lemma.
\end{proof}

\begin{lemma}\label{lem:deltamultiplicative}
 The map $\delta_*$ in Proposition~\ref{prop:longexact} is multiplicative with respect to the quantum product $*$.
In fact, the chain map $\delta$ satisfies $\delta(x*y)=\delta(x)*\delta(y)$.
\end{lemma}
\begin{proof}
 Recall that the long exact sequence in Proposition~\ref{prop:longexact} comes from the short exact sequence
\begin{equation*}
\xymatrix{
 0 \ar[r] & C(U;\tilde{f}|_U,\tilde{J})\ar[r]^{j} &
C((V,S);\tilde{f},\tilde{J})\ar[r]^{\delta} &
C(S^{\epsilon/2};\tilde{f}|_{S^{\epsilon/2}},J)\ar[r] & 0. }
\end{equation*}
The map $\delta:C((V,S);\tilde{f},\tilde{J})\rightarrow C(S^{\epsilon/2};\tilde{f}|_{S^{\epsilon/2}},J)$ is given by restricting the critical points of $C((V,S);\tilde{f},\tilde{J})$ to the critical points in $C(S^{\epsilon/2};\tilde{f}|_{S^{\epsilon/2}},J)$.
Recall that $x * y:= \smashoperator{\sum\limits_{\delta_{prod}(x,y,z,\lambda)=0}} \sharp \mathcal{P}_{prod}(x,y,z,\lambda)  zt^{\overline{\mu}(\lambda)}$.
Note that for any point $z\in C(S^{\epsilon/2};\tilde{f}|_{S^{\epsilon/2}},J)$ the space $\mathcal{P}_{prod}(x,y,z,\lambda)$ is non-empty if and only if $x,y\in C((V,S);\tilde{f},\tilde{J})\cap S^{\epsilon/2}$.
Indeed, the $\tilde{J}$-holomorphic curves which are not completely contained in $U$ have to be constant under $\pi$ by Lemma~\ref{lem:curve_openmapping}.
No flow line of $-\nabla \tilde{f}$ can go from $U$ to $S^{\epsilon/2}$ by the construction of $\tilde{f}$.
By Lemma~\ref{lem:orientationmatch} we conclude that also the signs of $\mathcal{P}_{prod}(x,y,z,\lambda;\tilde{f})$ and $\mathcal{P}_{prod}(\delta(x),\delta(y),\delta(z),\lambda;\tilde{f}|_{S^{\epsilon/2}}))$ match and thus we have $\delta(x* y)= \delta(x)*\delta(y)$.
\end{proof}

\begin{lemma}\label{lem:prodandi}
 The product $*$ on $QH_*(\tilde{V})$ is trivial on the image of the inclusion $i_*$.
In other words, for any two elements $a$ and $b$ in $QH_*(S)$ we have that $i_*(a)*i_*(b)=0$.
\end{lemma}
\begin{proof}
 Recall that the map $i$ is the connecting homomorphism of the following short exact sequence
\begin{equation*}
\xymatrix{
0 \ar[r] & C_k(U;\tilde{f}|_U,\tilde{J};\mathcal{R}) \ar[r]^j \ar[d]^{d_U} & C_k((V,S);\tilde{f},\tilde{J};\mathcal{R}) \ar[r]^{\delta} \ar[d]^{d_{(V,S)}} & C_{k-1}(S^{\epsilon/2};\tilde{f}|_{S^{\epsilon/2}},J;\mathcal{R}) \ar[d]^{d_{S^{\epsilon/2}}} \ar[r] & 0\\
0 \ar[r] & C_{k-1}(U;\tilde{f}|_U,\tilde{J};\mathcal{R}) \ar[r]^{j} & C_{k-1}((V,S);\tilde{f},\tilde{J};\mathcal{R}) \ar[r]^{\delta} & C_{k-2}(S^{\epsilon/2};\tilde{f}|_{S^{\epsilon/2}},J;\mathcal{R}) \ar[r] & 0 \\
}
\end{equation*}
We choose a Morse function $\tilde{f}$ as in Definition~\ref{def:specialMF}.
In particular $-\nabla \tilde{f}$ is transverse to $\partial V|_{[R_-, R_+]\times \mathbb{R}}$ and points outside along $\partial V|_{[R_-, R_+]\times \mathbb{R}}$.
Let $p$ be a cycle in $C_{k-1}(S^{\epsilon/2};\tilde{f}|_{S^{\epsilon/2}},J)$.
We consider the critical point $x_{{\iota},1}:=(t_{{\iota},1},a_{\iota},p)\in Crit(\tilde{f})\cap S^{\epsilon/2}$ of the whole function $\tilde{f}$.
As in the proof of Lemma~\ref{lem:sumofunits} we can show that $d(x_{{\iota},1})=x_{{\iota},0}$, where $x_{{\iota},0}:=(t_{{\iota},0},a_{\iota},p)\in Crit (\tilde{f})\cap S_{\frac{\epsilon}{4}}$.
Moreover, $x_{{\iota},0}$ is a cycle for the complex $C((V,S);\tilde{f},\tilde{J})$.
Indeed, the critical point $p$ is a cycle of $f_k$ and $t_{{\iota},0}$ is a minimum of the function $\sigma_{\iota}$.
By the definition of the connecting homomorphism it follows that $i(p)=x_{{\iota},0}$ and thus $i$ can be seen as the inclusion of the critical points of $C_{k-1}(S^{\epsilon/2};\tilde{f}|_{S^{\epsilon/2}})$ into $C(U;\tilde{f}|_U,\tilde{J})$.

It is left to show that the product on the chain level for $C(U;\tilde{f}|_U,\tilde{J})$ is zero for any two critical points in $S^{\epsilon/4}$. The product is defined using three Morse functions $\tilde{f}$, $\tilde{f}'$ and $\tilde{f}''$ on $U$ in general position.
In particular the stable and unstable manifolds of $\tilde{f}$ and $\tilde{f}'$ have to be transverse.
We may choose $\tilde{f}=\tilde{f}''$.
We assume that $\tilde{f}$ is a function as in Definition~\ref{def:specialMF}.
Then we can choose $\tilde{f}'$ such that the functions $f_i^{-}$, ${f'}_i^{-}$ and $f_j^+$, ${f'}_j^{+}$ are transverse respectively. 
However for $i\in I_-$ and $j\in J_+$ the critical points of $\sigma^-_i$ and $\sigma^+_j$ are slightly perturbed to obtain ${\sigma'}^-_i$ and ${\sigma'}^+_j$.
Assume that ${\sigma'}^+_j$ still has a unique maximum at $R_+ +\epsilon/2$ but the unique minimum is perturbed to lie at $R_+ +\epsilon/8$.
Similarly, ${\sigma'}^-_i$ has a unique maximum at $R_- - \epsilon/2$ and a unique minimum at $R_- - \epsilon/8$.
If $i\notin I_-$ and $j\notin J_+$ we just assume that $\sigma^-_i$, ${\sigma'}^-_i$ and ${\sigma}^+_j$, ${\sigma'}^+_j$ are transverse.
Let $x$ be a critical point in $C(U;\tilde{f}|_U,\tilde{J})$ that lies on $S^{\epsilon/4}$ and $y$ a critical point in $C(U;\tilde{f'}|_U,\tilde{J})$ that lies in $S^{\epsilon/8}$.
Any trajectory of $-\nabla \tilde{f}'$ starting at a critical point in $S^{\epsilon/8}$ stays in $S^{\epsilon/8}$, by the choice of $\sigma_i'$ and $\sigma_j'$. Similarly, $-\nabla \tilde{f}$ has to stay in $S^{\epsilon/4}$.
By Lemma~\ref{lem:curve_openmapping} all the $\tilde{J}$-holomophic disks involved in a figure-$Y$ pearly trajectory have to map to a constant point under $\pi$.
Thus, the moduli space of figure-$Y$ pearly trajectories starting at $x$ and $y$ is empty and therefore their product is zero.
\end{proof}

\begin{lemma}\label{lem:jandprod}
We have the following relation
\begin{equation*}
 j_*(x*y)=j_*(x)*j_*(y).
\end{equation*}

\end{lemma}
\begin{proof}
 The map $j:C(U;\tilde{f}|_U, \tilde{J})\rightarrow C((V,S);\tilde{f},\tilde{J})$ on the chain level is given by the inclusion. 
Any figure-$Y$ pearly trajectory in $C(U,\tilde{f}|_U, \tilde{J})$ is also a trajectory in $C((V,S);\tilde{f},\tilde{J})$. 
By the construction of the function $\tilde{f}$ in Definition~\ref{def:fadaptedS}, there exists no figure-$Y$ pearly trajectory starting at two generators in $C(U;\tilde{f}|_U, \tilde{J})$ and ending at a generator in $C((V,S);\tilde{f},\tilde{J})\setminus C(U;\tilde{f}|_U, \tilde{J})$.
\end{proof}

\begin{lemma}\label{lem:inc&deltacommute}
 The diagram
\begin{equation*}
\xymatrix{
  C_i(V, \partial V) \ar[r]^(0.45){i_{(V, \partial V)}} \ar[d]^{\delta}& C_{i-1}(\tilde{M}_T, \partial \tilde{M}_T) \ar[d]^{\delta}\\
 C_i(\partial V) \ar[r]^{i_{\partial V}} & C_{i-1}(\partial \tilde{M}_T)\\
}
\end{equation*}
is commutative, where $i_{\partial V}:=\oplus i_{L_i^+}\oplus i_{L_j^-}$.
\end{lemma}

\begin{proof}[Proof of lemma~\ref{lem:inc&deltacommute}]
We choose a Morse function $f$ on $M$ and a special Morse function $\tau$ on $\mathbb{R}^2$ with with the property that $\tilde{g}:=\tau+ f$ fulfills the necessary requirements to define the quantum homology of $\tilde{M}_T$ (see section~\ref{sec:ModuleStructure}).
Outside of the compact set $K\subset \mathbb{R}^2$ the function $\tau$ has two critical points $(R_- - \epsilon,0)$ and $(R_+ + \epsilon,0)$ of index $1$.
We may assume that the unstable manifolds of $(R_- - \epsilon,0)$ and $(R_+ + \epsilon,0)$ are vertical, such that all together the negative gradient of $\tau$ points outside at the boundary of $T$ and inside at its complement.
Then the map $\delta: C_i(\tilde{M}_T, \partial \tilde{M}_T) \rightarrow C_{i-1}(\partial \tilde{M}_T)$ is given by the restriction of the generators of $\tilde{g}$ to the generators in the fiber of $(R_- - \epsilon,0)$ and $(R_+ + \epsilon,0)$ under the projection $\pi$.
For any $x\in Crit(\tilde{f})\setminus Ker(\delta)$ and any $a\in Crit(\tilde{g})\setminus Ker(\delta)$ we have a bijection between the elements in the space $\mathcal{P}_{inc}(\delta(x),\delta(a),\lambda;\tilde{f},\tilde{g})$ and $\mathcal{P}_{inc}(x,a,\lambda;\tilde{f}|_{S^{\epsilon/2}},\tilde{g}|_{(R_- - \epsilon,0)\cap (R_+ + \epsilon,0)})$.
This follows by an analogous argument as in the previous lemma.
\end{proof}

\subsection{Module Structures in the Long Exact Sequence}

Once again consider the long exact sequence
\begin{equation*}
\xymatrix{
 \dots \ar[r]^{\delta_*} & QH_*(\partial V) \ar[r]^{i_*} & QH_*(V)\ar[r]^{j_*} & QH_*(V,\partial V) \ar[r]^{\delta_*} & QH_{*-1}(\partial V) \ar[r]^{i_*} &\dots
}
\end{equation*}
Notice that each of the quantum homologies in this sequence is a module over some ambient quantum homology. 
We would like to examine the compatibility of the quantum products with these module structures.
Consider the sequence
\begin{equation*}
\xymatrix{
 \dots \ar[r] & QH_*(M) \ar[r]^(0.4){\Phi} & QH_{*+2}(\tilde{M}_R,\partial \tilde{M}_R) \ar[r]^{ id} & QH_{*+2}(\tilde{M}_R,\partial \tilde{M}_R) \ar[r]^(0.6){\Phi^{-1}} & QH_*(M) \ar[r] & \cdots \\
}
\end{equation*}
where $\Phi$ is the ring isomorphism $\Phi:QH_*(M) \rightarrow QH_{*+2}(\tilde{M}_R,\partial \tilde{M}_R) : b \mapsto R\times b$.
Then we have the following lemma:
\begin{lemma}
 The module structures on $QH_*(\partial V)$,  $QH_*(V)$ and $QH_*(V,\partial V)$ over the rings $QH_*(M)$, $QH_{*+2}(\tilde{M}_R,\partial 
\tilde{M}_R)$ and $QH_{*+2}(\tilde{M}_R,\partial \tilde{M}_R)$ respectively fulfill the following identities:
\begin{enumerate}
 \item[(i)] $i_*(a \ast x)= \Phi(a) \ast i_*(x)$, for every $x\in QH(\partial V)$ and every $a\in QH(M)$.
 \item[(ii)] $j_*(a\ast x)=a\ast j_*(x)$ for every $x\in QH(V)$ and every $a\in QH(\tilde{M}_R,\partial \tilde{M}_R)$.
 \item[(iii)] $\delta_*(a\ast x)=\Phi^{-1}(a)\ast \delta_*(x)$, for every $x\in QH(V,\partial V)$ and every $a\in QH(\tilde{M}_R,\partial \tilde{M}_R)$.
\end{enumerate}
i.e. the maps in the long exact sequence are module maps.
\end{lemma}

\begin{proof}
\begin{enumerate}
 \item[(i)]We choose a Morse function $\tilde{f}$ on $V$ as in Definition~\ref{def:specialMF} and a Morse function $\tilde{h}:=\tau_R + h$ on $\tilde{M}$ as in Section~\ref{sec:ModuleStructure}.
Then the coefficient of $a\ast x$ in front of $y t^{\overline{\mu}(\lambda)}$ for some critical point $y$ of $\tilde{f}|_{\partial V ^{\epsilon/2}}$ is defined by counting the elements in the zero dimensional moduli spaces $\mathcal{P}_{mod}(x,y,a,\lambda;\tilde{f}|_{\partial V ^{\epsilon/2}},h)$.
Recall from the proof of Lemma~\ref{lem:prodandi} that for such a function $\tilde{f}$ the inclusion $i$ of some cycle $x\in C(\partial V^{\epsilon/2}; f, J)$ is given by $i(x)=(t_{{\iota},0},a_{\iota},x)$ for some ${\iota}\in I_-\cup J_+$.
We need to compare the elements of $\mathcal{P}_{mod}(x,y,a,\lambda;\tilde{f}|_{\partial V ^{\epsilon/2}},h)$ and the elements of $\mathcal{P}_{mod}((t_{{\iota},0},a_{\iota},x),z,(\xi, a),\lambda;\tilde{f},\tilde{h})$ for some critical points $z\in C(U; \tilde{f}|_U, \tilde{J})$ and $\xi\in Crit(T_R)$.
By the same arguments as in the proof of Lemma~\ref{lem:prodandi} we know that the pearly trajectory in $V$ starting at the critical point $i(x)=(t_{{\iota},0},a_{\iota},x)\in U$ of $\tilde{f}$ cannot leave the set $\partial V_{\epsilon/4}$.
Hence $\mathcal{P}_{mod}((t_{{\iota},0},a_{\iota},p),z,({\xi}, a),\lambda;\tilde{f},\tilde{h})$ is not empty if and only if $z=(t_{{\iota'},0},a_{\iota'},p)$ for some $p\in im(i)$ and ${\iota}={\iota}'$.
Moreover there exists a unique flow line of $-\nabla \tau_R$ form the unique maximum ${\xi}$ of $\tau_R$ to the any point in $R$, in particular the point $(t_{{\iota},0},a_{\iota})$.
This implies that $\sharp \mathcal{P}_{mod}((t_{{\iota},0},a_{\iota},x),(t_{{\iota},0},a_{\iota},p),(\xi, a),\lambda;\tilde{f},\tilde{h})=\sharp \mathcal{P}_{mod}(x,p,a,\lambda;f,h)$, which proves the statement.
 \item[(ii)] 
Recall that the map $j$ on the chain level is the inclusion of the subcomplex $C_*(U;\tilde{f}|_U,\tilde{J};\mathcal{R})$ into the complex $C_*((V,S);\tilde{f},\tilde{J};\mathcal{R})$.
The expression $j_*(a\ast x)$ counts elements of the moduli space $\mathcal{P}_{mod}(x,y,a,\lambda;\tilde{f},\tilde{h})$, where $x$ and $y$ are both contained in $Crit(\tilde{f})\cap U$ and $a\ast j_*(x)$ counts elements in the moduli space $\mathcal{P}_{mod}(x,y,a,\lambda;\tilde{f},\tilde{h})$, where $x\in im(j)=Crit(\tilde{f})\cap U$ and $y\in Crit(\tilde{f})$.
There exists a bijection between these moduli spaces.
Indeed, by the same argument as in the previous lemma, the moduli space $\mathcal{P}_{mod}(x,y,a,\lambda;\tilde{f},\tilde{h})$ is empty if $y \notin Crit(\tilde{f})\cap U $.
 \item[(iii)]
Recall that the map $\delta$ on the chain level, is defined by restricting $Crit(\tilde{f})$ to $Crit(\tilde{f}|_{S^{\epsilon/2}})$.
The expression $\delta_*(a\ast x)$ is defined by counting the elements in the moduli spaces of the type $\mathcal{P}_{mod}(x,y,a,\lambda;\tilde{f},\tilde{h})$ with $x\in Crit(\tilde{f})\cap S^{\epsilon/2}$, $a\in  Crit(\tilde{h})$ and $y=(t_{{\iota},1},a_{\iota},q)\in Crit(\tilde{f})\cap S^{\epsilon/2}$.
As before, the critical point $x$ has to lie in $S^{\epsilon/2}$, since otherwise there exists no pearly trajectory from $x$ to $y$ by the choice of $\tilde{f}$. 
Notice also that this space is empty whenever $y$ is not contained in the same connected component of the boundary as $x$.
Therefore we can write $x:=(t_{{\iota},1},a_{\iota},p)$, where $p\in Crit(\tilde{f}|_{S^{\epsilon/2}})$.
To compute $\Phi^{-1}(a)\ast \delta_*(x)$ we count elements in the moduli space of the form
$\mathcal{P}_{mod}(p,q,\pi_M(a),\lambda;\tilde{f}|_{S^{\epsilon/2}},\tilde{f'}|_{S^{\epsilon/2}},h)$, where $q$ is some critical points $\tilde{f}|_{S^{\epsilon/2}}$ and $\pi_M$ denotes the projection $\pi_M:\mathcal{R}\times M \rightarrow M$.
This space is empty if $p$ and $q$ are not contained in the same connected component of $\partial V^{\epsilon}$.
Clearly there exists a unique negative gradient flow line of $\tau_R$ from $\pi(a)$ to $\pi(x)=\pi(y)=(t_{{\iota},1},a_{\iota})$.
This shows that there is a bijection between the moduli space $\mathcal{P}_{mod}(x,y,a,\lambda;\tilde{f},\tilde{h})$ and $\mathcal{P}_{mod}(p,q,\pi_M(a),\lambda;f,h)$, which proves the last part of the lemma.
\end{enumerate}

\end{proof}

\section{An Example}\label{sec:example}

In this section we compute the quantum homologies for a simple example.
Let $L_1$ and $L_2$ be two Lagrangian spheres intersecting at exactly one point.
Applying Lagrangian surgery as defined by Polterovich in~\cite{polterovich} gives 
a Lagrangian cobordism $V$ with negative ends $L_1$ and $L_2$ and one positive end 
$L_3:=L_1\sharp L_2$, which is the sphere obtained from $L_1$ and $L_2$ by surgery. 
See~\cite{LagrCobordism} or~\cite{cedric} for a precise explanation.
Moreover we can endow $V$ with a spin structure (see~\cite{cedric} and 
Section~\ref{sec:orientation}). The ends $L_1$, $L_2$ and $L_3$ of $V$ inherit spin 
structures too. 
Hence, once a spin structure is fixed on $V$ we can work over any ring.
In order to compute the quantum homology for this cobordism, we first start by computing 
the singular homology.

The cobordism $V$ is topologically $int B^3\setminus ( int B_1^3 \coprod int B_2^3)$, 
where $B^3$ denotes the $3$-dimensional ball.
It therefore has the homotopy type of $S^2\lor S^2$.
Therefore the singular homology groups are:
\begin{center}
  \begin{tabular}{ c || c | c | c | c }
    $*$ & $0$ & $1$ & $2$ & $3$ \\ \hline
    $H_*(V)$ & $\mathbb{Z}$ & $0$ & $\mathbb{Z}\oplus\mathbb{Z}$ & $0$ \\ 
    $H_*(V,\partial V)$ & $0$ & $\mathbb{Z}\oplus\mathbb{Z}$ & $0$ & $\mathbb{Z}$. \\
  \end{tabular}
\end{center}
The differential $d$ for the quantum homology can be written as the sum 
\begin{equation*}
 d=d_0+d_1 t+ d_2 t^2 + \dots + d_i t^i+ \dots ,
\end{equation*}
where $d_i$ has degree $-1+i N_V$. 
Since $N_V=2$ this implies that $d_i=0$ unless $i=0$ and in this case $d_0$ is 
the usual Morse differential.
We conclude that there exist isomorphisms
\begin{equation*}
 QH_*(V)\cong H_*(V)\otimes \Lambda
\end{equation*}
and 
\begin{equation*}
 QH_*(V,\partial V)\cong H_*(V, \partial V)\otimes \Lambda
\end{equation*}

Note that these isomorphisms are not canonical in the sense that there are no preferred 
such isomorphisms but rather a preferred class of them, 
see~\cite{LQH},~\cite{rigidity} and~\cite{LagrTop} for more details on that.
Similarly, let us compute the quantum homologies of the pair $(V,S)$, where $S\subset 
\partial V$ 
is a collection of connected components of the boundary $\partial V$.
Suppose that $S=L_i$ for some $i=1,2,3$, then $H_*(S)$ is equal to $\mathbb{Z}$ in degree 
zero and 
two, and zero otherwise. 
The long exact sequence
\begin{equation*}
\xymatrix{
 \dots \ar[r]^{\delta_*} & H_*(S) \ar[r]^{i_*} & H_*(V)\ar[r]^{j_*} & 
H_*(V,S)\ar[r]^{\delta_*} & H_{*-1}(S)\ar[r]^{i_*} & \dots,
}
\end{equation*}
becomes in our case the sequence
\begin{equation*}
\xymatrix{
 \dots \ar[r]^{\delta_*} & 0 \ar[r]^{i_*} & 0\ar[r]^{j_*} & 
H_3(V,S)\ar[r]^{\delta_*} &\dots \\
\dots \ar[r]^{\delta_*} & \mathbb{Z} \ar[r]^{i_*} & 
\mathbb{Z}\oplus\mathbb{Z} \ar[r]^{j_*} & 
H_2(V,S) \ar[r]^{\delta_*} &\dots \\
\dots \ar[r]^{\delta_*} & 0 \ar[r]^{i_*} & 
0 \ar[r]^{j_*} & 
H_1(V,S) \ar[r]^{\delta_*} &\dots \\
\dots \ar[r]^{\delta_*} & \mathbb{Z} \ar[r]^{i_*} & 
\mathbb{Z}\ar[r]^{j_*} & 
H_0(V,S) \ar[r]^{\delta_*} & 0  \\
 }.
\end{equation*}
Since we know that $H_0(V,S)=H_3(V,S)=0$ we conclude that 
\begin{center}
  \begin{tabular}{ c || c | c | c | c }
    $*$ & $0$ & $1$ & $2$ & $3$ \\ \hline
    $H_*(V,L_i)$ & $0$ & $0$ & $\mathbb{Z}$ & $0$ \\ 
    $H_*(V,\partial V\setminus L_i)$ & $0$ & $\mathbb{Z}$ & $0$ & $0$ \\
  \end{tabular}
\end{center}
By the same argument as before we have isomorphisms $QH_*(V,L_i)\cong H_*(V,L_i)\otimes 
\Lambda$ and
$QH_*(V,\partial V \setminus L_i)\cong H_*(V,\partial V \setminus L_i)\otimes \Lambda$.

The next lemma defines a useful basis for $QH_*(V,\partial V)$ over the ground ring 
$\mathcal{R}=\mathbb{Q}$.
\begin{lemma}\label{lem:canonicalbasis}
The elements
 \begin{equation}\label{xi}
 x_i:= ([L_i]\times R)*e_{(V,\partial V)}.
\end{equation} 
for $i=1,2$ together with the unit $e_{(V,\partial V)}$ form a basis of $QH_*(V,\partial 
V)$ over $\Lambda:=\mathbb{Q}[t,t^{-1}]$.
\end{lemma}

\begin{proof}
 It suffices to show that $x_1$ and $x_2$ are linearly independent over $\Lambda$.
Let $\pi_i:\partial V\rightarrow L_i$ denote the projection of $\partial V$ to $L_i$.
Notice that 
\begin{equation*}
 \begin{array}{ccc}
  \epsilon(\pi_i(\delta(x_j))&=&\epsilon([L_j]*e_{L_i})\\
  &=&[L_j]\cdot [L_i]\\
\end{array}
\end{equation*}
and therefore
\begin{equation*}
\begin{array}{ccccc}
  \epsilon(\pi_1(\delta(x_1))&=& -2\\
  \epsilon(\pi_1(\delta(x_2))&=& 1\\
  \epsilon(\pi_2(\delta(x_1))&=& 1\\
  \epsilon(\pi_2(\delta(x_2))&=& -2,\\
\end{array}
\end{equation*}
which proves the lemma.
\end{proof}

A basis for $QH_*(\partial V)$ is given by $e_{L_1}, \overline{p_1}, e_{L_2}, 
\overline{p_2},e_{L_3}, \overline{p_3}$, where $\overline{p_i}$ is a lift of the class of 
a point in $H_0(L_i)$ under the augmentation.

A basis for $QH_*(V)$ is $\overline{p}$, $[L_1]_{QH_*(V)}$, $[L_2]_{QH_*(V)}$, where 
$\overline{p}$ is the lift of a point under the augmentation and 
$[L_1]_{QH_*(V)}$, $[L_2]_{QH_*(V)}$ are the classes corresponding 
to the elements $[L_1],[L_2]\in H_2(V)$ in the following sense.
There exists a short exact sequence 
\begin{equation}\label{eq:ttozero}
 \begin{array}{ccccccccc}
  0 & \rightarrow & H_2(V)t & \rightarrow & QH_0(V) & \rightarrow & H_0(V) & \rightarrow 
& 0\\
 \end{array}.
\end{equation}
Consider the map
\begin{equation}\label{eq:isoincl}
 H_2(\partial V)t\rightarrow  H_2(V)t
\end{equation}
induced by the inclusion.
Denote by $[L_1]_{QH_*(V)}\in QH_*(V) $, $[L_2]_{QH_*(V)}\in 
QH_*(V)$ the image of the classes $[L_1]$ and $[L_2]\in H_2(\partial V)$ under the 
map~(\ref{eq:isoincl}) and the first map in the short exact sequence~(\ref{eq:ttozero}).

The following lemma helps us to compute the maps $\delta_*$, $i_*$ and $j_*$.
\begin{lemma}
 Denote by $[(V,\partial V)]$ the image of the unit $e_{(V,\partial V)}$ under the 
inclusion $i_{(V,\partial V)}$.
Then
\begin{equation}
[(V,\partial V)]=([L_1]_{QH_*(V)}+ [L_2]_{QH_*(V)})\times I =- [L_3]_{QH_*(V)}\times I.
\end{equation}
\end{lemma}

\begin{proof}
 Consider a Morse function $\tilde{f}$ on $V$ respecting the exit region $\partial V$. 
We may assume that $\tilde{f}$ has four local 
maxima.
One inside the set $U$ (see Definition~\ref{def:specialMF}) and one for 
each of the ends $L_i$. 
Moreover, we
assume that the local maxima on the negative ends of $V$ have the same $x$-coordinate. 
Let $\xi_1$ denote the $x$-coordinate of the maxima on the negative ends, $\xi_2$ the 
$x$-coordinate of the critical point in the set $U$ and
$\xi_3$ the $x$-coordinate of the maxima on the 
positive end.
We introduce three special Morse functions $\tilde{g}_i$ on 
$\tilde{M}_T$ for $i=1,2,3$. 
Fix a Morse function $f$ on $M$ and let $\tilde{g}_i=(\tau_T)_i+ f$ be the Morse 
function on $\tilde{M}_T$, such that $(\tau_T)_i$ has exactly one critical point with 
$x$-coordinate $\xi_i$.
Taking the inclusion by counting the appropriate moduli spaces with respect to these 
three functions shows that $[(V,\partial V)]=([L_1]_{QH_*(V)}+ [L_2]_{QH_*(V)})\times I 
= 
-[L_3]_{QH_*(V)}\times I$.
\end{proof}

The maps $\delta_*$, $i_*$ and $j_*$ act on the basis elements defined above as follows.
\begin{equation*}
\begin{array}{cccc}
  \delta_*: & QH_*(V,\partial V) & \rightarrow & QH_{*-1}(\partial V)\\
 & e_{(V,\partial V)} & \mapsto & e_{L_1}+e_{L_2}+e_{L_3}\\
 & ([L_1]\times R)*e_{(V,\partial V)}  & \mapsto &[ 
L_1]*(e_{L_1}+e_{L_2}+e_{L_3})\\
 & ([L_2]\times R)*e_{(V,\partial V)}  & \mapsto & 
[L_2]*(e_{L_1}+e_{L_2}+e_{L_3})\\
\end{array}
\end{equation*}

\begin{equation*}
\begin{array}{cccc}
  i_*: & QH_*(\partial V) & \rightarrow & QH_{*}(V)\\
 & e_{L_1} & \mapsto & [L_1]_{QH_*(V)}\\
 & e_{L_2} & \mapsto & [L_2]_{QH_*(V)}\\
 & e_{L_3} & \mapsto & -[L_1]_{QH_*(V)}-[L_2]_{QH_*(V)}\\
 & \overline{p_i} & \mapsto &  \overline{p} \ \quad \forall i\\
\end{array}
\end{equation*}

\begin{equation*}
\begin{array}{cccc}
  j_*: & QH_*(V) & \rightarrow & QH_{*}(V,\partial V)\\
 & [L_1]_{QH_*(V)}  & \mapsto & 0\\
 & [L_2]_{QH_*(V)}  & \mapsto & 0\\
 & \overline{p} & \mapsto & 0\\
\end{array}
\end{equation*}

On the basis elements described above the inclusion map from $QH_*(V,\partial V)$ into 
the ambient quantum homology 
$QH_*(\tilde{M}_T,\partial \tilde{M}_T)$ is given by

\begin{equation*}
 \begin{array}{cccc}
  i:& QH_*(V,\partial V)& \rightarrow & QH_*(\tilde{M}_T,\partial \tilde{M}_T)\\
 & ([L_1]\times R)*e_{(V,\partial V)}  & \mapsto & [L_1]*([L_1]+[L_2])\times I\\
 & ([L_2]\times R)*e_{(V,\partial V)}  & \mapsto & [L_2]*([L_1]+[L_2])\times I\\
 &e_{(V,\partial V)} & \mapsto & ([L_1]+[L_2])\times I\\
 \end{array}
\end{equation*}

Next, we compute the quantum product structures.
Recall that a basis of $QH_*(V,\partial V)$ is given by
\begin{center}
  \begin{tabular}{ c || c | c | c | c }
    $*$ & $0$ & $1$ & $2$ & $3$ \\ \hline
    $QH_*(V,\partial V)$ & & $x_1:=([L_1]\times R)*e_{(V,\partial V)}$ & & 
$e:=e_{(V,\partial V)}$  \\ 
    & & $x_2:=([L_2]\times R)*e_{(V,\partial V)}$ & &  \\ 
  \end{tabular}
\end{center}
where $e$ is the unit.
Clearly
\begin{equation*}
\begin{array}{ccc}
 x_1*x_1 &=& ([L_1]*[L_1]\times R)*e\\
 x_2*x_2 &=& ([L_2]*[L_2]\times R)*e\\
 x_1*x_2 &=& ([L_1]*[L_2]\times R)*e.\\
\end{array}
\end{equation*}
Recall that here the product $[L_i]*[L_j]$ denotes the product in the ambient quantum 
homology, which was defined using Gromov-Witten invariants. See Section~\ref{sec:QHandFH} 
and~\cite{McDuffSalamon} for more details.

\begin{remark}
 There is a more explicit way to calculate the product of the elements from 
Lemma~\ref{lem:canonicalbasis}, which in fact leads to a proof that the discriminats 
(see~\cite{cedric}) of 
the Lagrangian spheres $L_1$, $L_2$ and $L_3$ in the example above are all the same.
This is a result that has already been proven in a different way in~\cite{cedric}.
Moreover, if there is a configuration of Lagrangian spheres with Intersection graph as in 
figure~\ref{fig:A1}
\begin{figure}
 \centering
\includegraphics{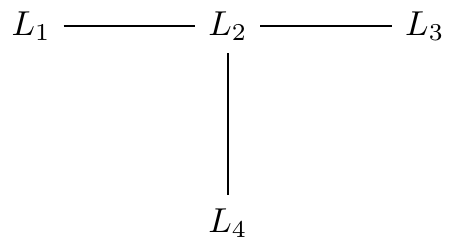}
\caption{Intersection graph of Lagrangian spheres}\label{fig:A1}
\end{figure}
then we can show that the discriminats of all $L_i$ are zero.
This will be further explored in~\cite{discriminats}.
\end{remark}

\subsection*{Acknowledgements}
This paper was written in the framework of my master thesis, supervised by Prof. 
Paul Biran at ETH Zürich.
I am grateful to Paul Biran for his guidance throughout this paper.
I would like to thank Felix Hensel and Will Merry for their help and comments.
For helpful discussions I would also like to thank Cedric Membrez. 

\bibliographystyle{plain}
\bibliography{QH.bib}

 \end{document}